\documentclass[pdflatex,sn-mathphys]{sn-jnl}% Math and Physical Sciences Numbered Reference Style
\usepackage{graphicx}%
\usepackage{multirow}%
\usepackage{amsmath,amssymb,amsfonts}%
\usepackage{amsthm}%
\usepackage{mathrsfs}%
\usepackage[title]{appendix}%
\usepackage{xcolor}%
\usepackage{textcomp}%
\usepackage{manyfoot}%
\usepackage{booktabs}%
\usepackage{algorithm}%
\usepackage{algorithmicx}%
\usepackage{algpseudocode}%
\usepackage{listings}%
\theoremstyle{thmstyleone}%
\newtheorem{theorem}{Theorem}%  meant for continuous numbers
\newtheorem{proposition}[theorem]{Proposition}% 

\newtheorem{corollary}{Corollary}

\newtheorem{lemma}{Lemma}

\theoremstyle{thmstyletwo}%
\newtheorem{remark}{Remark}%

\theoremstyle{thmstylethree}%
\newtheorem{definition}{Definition}%

\usepackage{placeins}

\usepackage{array}

\allowdisplaybreaks

\begin{document}

\title[A description of the Laguerre-Hahn orthogonal polynomials of class zero revisited]{A description of the Laguerre-Hahn orthogonal polynomials of class zero revisited}


\author[1]{\fnm{Mohamed} \sur{Khalfallah}}\email{mohamed.khalfallah@fsm.rnu.tn}

\author[2]{\fnm{Pascal} \sur{Maroni}} %\email{}
\equalcont{17th January 1933 -- 16th January  2024}

\author*[3]{\fnm{Zélia} \sur{da Rocha}}\email{mrdioh@fc.up.pt}

\affil[1]{\orgdiv{Department of Mathematics}, 
\orgname{Faculty of Sciences of Monastir, University of Monastir}, 
\orgaddress{
%\street{}, 
\city{Monastir}, 
\postcode{5019}, 
%\state{}, 
\country{Tunisia}}}

\affil[2]{\orgdiv{Laboratoire Jacques-Louis Lions}, 
\orgname{Sorbonne Université, CNRS}, 
\orgaddress{
\street{Boite courrier 187; 4, place Jussieu}, 
\city{Paris}, 
\postcode{75252 Paris cedex 05}, 
%\state{}, 
\country{France}}}

\affil[3]{\orgdiv{Departamento de Matemática, Centro de Matemática da Universidade do Porto (CMUP)}, 
\orgname{Faculdade de Ciências da Universidade do Porto}, 
\orgaddress{
\street{Rua do Campo Alegre n. 687}, 
\city{Porto}, 
\postcode{4169-007}, 
%\state{}, 
\country{Portugal}}}


\abstract{This paper has a threefold aim. On the one hand, we provide a complete description of Laguerre–Hahn forms of class zero. This fills a gap in the literature: more precisely, up to an affine change of variables, there are ten families, including two new ones analogous to the classical Bessel family. On the other hand, we establish the structure relations for all these families, correcting those that have previously been reported in the literature. At last, as an application, using an algorithm recently obtained, with the aid of symbolic computations, we derive four new structure relations and a new fourth-order differential equation for one of the new families analogous to Bessel.}




\maketitle

%%% ----------------------------------------------------------------------
%\maketitle
%%% ----------------------------------------------------------------------
%\tableofcontents%%%%%%%%%%%%%%%%%%%%%%%%%%%%%%%%
%\tableofcontents

\newpage
%%%%%%%%%%%%%%%%%%%%%%%%%%%%%%%%%%%
\section{Introduction}
%%%%%%%%%%%%%%%%%%%%%%%%%%%%%%%%%%%

The theory of orthogonal polynomials related to Stieltjes functions $S$ satisfying 
a Riccati equation with polynomial coefficients of the  type
\begin{equation}\label{intro-Riccati}
\Phi S' = B S^2 + C S + D, \quad \Phi \neq 0,
\end{equation}
has a long and rich 
history, rooted in the framework of Pad\'e rational approximation and Jacobi continued fractions associated
with analytic functions at infinity \cite{Rebocho-2026}.
The set of orthogonal polynomial sequences arising in this context is known as the 
Laguerre-Hahn families~\cite{Magnus-1983}.
An algebraic and topological framework for Laguerre-Hahn linear functionals 
was developed in \cite{Dini-these-1988, Maroni-1991}, establishing in 
particular, the equivalence between the above Riccati equation \eqref{intro-Riccati} satisfied by the formal 
Stieltjes function
\begin{equation*}
S(z):=-\sum_{n\ge0}\frac{(u)_n}{z^{n+1}},
\end{equation*}
where $(u)_n$ denotes the $n$-th moment of the corresponding linear functional $u$, 
and the following functional equation
\begin{equation}\label{intro-Riccati2}
(\Phi u)' + \psi u + B(x^{-1}u^2)=0, \quad \psi = -\Phi' - C.
\end{equation}
This result constitutes one of the cornerstones of the theory. A structure relation 
characterizing Laguerre-Hahn polynomials was also established in 
%\cite[Chapter~IV, Theorem~1.1]{Dini-these-1988}; 
\cite{Dini-these-1988};
more precisely, each Laguerre–Hahn polynomial 
$P_n$, $n \geq 0$, satisfies the so-called structure relation
\begin{equation*}
\Phi(x)P_{n+1}^{\prime}(x) - B(x)P_n^{(1)}(x) = \sum_{\mu=n-s}^{n+d} 
\theta_{n,\mu} P_\mu(x), \quad n \geq s + 1,
\end{equation*}
where $\Phi$ and $B$ are the polynomials defined in \eqref{intro-Riccati2},  
$\{P_n^{(1)}\}_{n \geq 0}$ denotes the sequence of associated orthogonal 
polynomials of order~$1$ of $\{P_n\}_{n \geq 0}$, and $d=\max(t,q)$, $s=\max(p-1,d-2)$, being $t$, $p$, and $q$ the degrees of  $\Phi$, $\psi$, and $B$, respectively. In particular, 
$\{P_n\}_{n \geq 0}$ satisfies the following main structure relation
\begin{align*}
\Phi(x) P_{n+1}'(x) - B_0(x) P_n^{(1)}(x) = \frac{1}{2}\big(C_{n+1}(x) 
- C_0(x)\big) P_{n+1}(x)- \gamma_{n+1} D_{n+1}(x) P_n(x), 
\end{align*}
for any $n\geq0$, where $C_n$ and $D_n$ are polynomials with coefficients depending on $n$, 
satisfying $
\deg C_n \leq s + 1,~\deg D_n \leq s$.\\

From a structural standpoint, the Laguerre-Hahn set occupies a central place 
in the theory of orthogonal polynomials, as it encompasses a wide range of polynomial 
families studied in the literature. Notably, it contains the semiclassical class as 
a special case, recovered when $B = 0$, while the classical families arise under the 
further conditions $\deg(\Phi) \leq 2$ and $\deg(\psi) = 1$ \cite{Maroni-1987}. 
Moreover, every linear functional of degree three belongs to the Laguerre-Hahn 
class \cite{Salah-Maroni-2000}, although the converse does not hold in general.

Laguerre-Hahn orthogonal polynomials can be constructed through various transformations 
of semiclassical orthogonal polynomials, either by modifying the corresponding Stieltjes 
function or by perturbing the recurrence coefficients 
\cite{Askey-1984, Dehesa-1990, Maroni-1991, Ronveaux-1990}. This approach gives rise 
to several well-known families, including associated, co-recursive, co-dilated, and 
co-modified polynomial sequences 
\cite{Belmehdi-1989, Belmehdi-Ronveaux-1991, Letessier-1994, Ronveaux-1990}. 
Despite this wealth of examples, a complete classification of Laguerre-Hahn orthogonal polynomials remains an open problem.
Moreover, the explicit determination of Laguerre-Hahn linear functionals is, in general, a highly non-trivial task. It should be noted that Laguerre-Hahn linear functionals of class zero have been investigated in~\cite{Bouakkaz-these,Bouakkaz-Maroni-1991}, both through the functional equation and the second-order recurrence relation. However, it should be pointed out that the resolution of system~\eqref{r0}--\eqref{Fn-System} carried out in~\cite{Bouakkaz-these,Bouakkaz-Maroni-1991} remains incomplete. Furthermore, the coefficients of the Laguerre-Hahn main structure relation associated with the eight families provided in~\cite{Bouakkaz-these} turn out to be incorrectly stated. \\

Recently, in \cite{khalfallah2026}, a general constructive method was developed to derive four new structure relations and a fourth-order linear differential equation satisfied by Laguerre-Hahn orthogonal polynomials. 
The approach employed was based on a systematic combination of structure relations, their derivatives, and algebraic elimination techniques, leading to a compact representation of the differential equation as a determinant of order 4. In the sequel, the coefficients of the differential equation were made explicit.
Moreover, the method was systematized in the algorithm {\it 4oDELH} (fourth-order differential equation for Laguerre-Hahn). Its implementation in the \textit{Mathematica}$^\circledR$ language has already produced the symbolic results for all Laguerre-Hahn families of class~0. To accomplish this task, the structure coefficients mentioned in the previous paragraph were required.
The results for the two cases analogous to Hermite were presented in \cite{khalfallah2026}. In this work, we provide, in Appendix~\ref{Section5}, the results corresponding to the second case analogous to Bessel. The remaining families will be addressed in forthcoming articles.\\

In the sequel of  \cite{khalfallah2026}, the aim of the present work is therefore threefold. On one hand, we carry out a complete resolution of system~\eqref{r0}--\eqref{Fn-System}, leading to two additional families analogous to Bessel that have not been reported elsewhere in the literature. On the other hand, we determine the coefficients of the Laguerre-Hahn structure relation for all ten resulting families, correcting in particular those that were stated in~\cite{Bouakkaz-these}. Last but not least, we provide four new structure relations and the fourth-order linear differential equation satisfied by one of the new families analogous to Bessel.\\

The remainder of this paper is organized as follows. Section~\ref{Section2} provides a concise summary of the background material needed throughout, covering the core definitions and principal results on orthogonal polynomials, moment functionals, and Laguerre-Hahn forms. In Section~\ref{Section3}, based on the system governed by the recurrence coefficients $\beta_n$ and $\gamma_{n+1}$ for a Laguerre-Hahn sequence of class zero, we establish a general result (Theorem~\ref{theorem-SR}) that provides the coefficients of the main structure relation. Section~\ref{Section4} is devoted to a complete description of Laguerre-Hahn forms of class zero via the functional equation and the second-order recurrence relation satisfied by such polynomials, as well as the coefficients of the main structure relation, by means of the general result obtained in the preceding section. More precisely, up to an affine change of variables, ten families are identified, including two new ones that are analogous to the Bessel case. Finally, Section~\ref{Section5} is dedicated to an application of the general theoretical framework developed in~\cite{khalfallah2026}. Specifically, we derive four additional structure relations and a fourth-order linear differential equation satisfied by the orthogonal polynomials corresponding to Case 2, analogous to Bessel. At last, we present some final remarks.

%\newpage
%%%%%%%%%%%%%%%%%%%%%%%%%%%%%%%%%%%
\section{Notation and basic background}\label{Section2} 
%%%%%%%%%%%%%%%%%%%%%%%%%%%%%%%%%%%

In this section, we present some basic definitions, notations, and results that are used throughout this paper.
%%%%%%%%%%%%%%%%%%%%%%%%%%%%%%%%%%%%%%%%%%%%%%%%%%%%%%%%%
\subsection{Basic tools}
Let $\mathcal{P}$ denote the vector space of polynomials with complex coefficients, and let $\mathcal{P'}$ be its algebraic dual space.  
The elements of $\mathcal{P'}$ will be referred to as \emph{forms} (or linear functionals).  
The pairing between $\mathcal{P}$ and $\mathcal{P'}$ is expressed through the duality brackets $\langle \cdot, \cdot \rangle$.  
For a form $u \in \mathcal{P'}$, the sequence of complex numbers $(u)_n,~ n \geq 0$, is called the \emph{moment sequence} of $u$ relative to the monomial basis $\{x^n\}_{n \geq 0}$.  
In particular, the $n$-th moment is given by $(u)_n := \langle u, x^n \rangle$, 
so that $u$ is uniquely determined by the sequence of its moments.  

In the following, we shall refer to a sequence $\{P_n\}_{n \geq 0}$ as a \emph{polynomial sequence} (PS) if $\deg P_n = n$ for all $n \geq 0$. A \emph{monic polynomial sequence} (MPS) is a PS in which each polynomial has a leading coefficient equal to one.  
If ${\{P_{n}\}}_{n \geq 0}$ is a MPS, there exists a unique sequence
$\{u_n\}_{n\geq 0}$, $u_n\in\mathcal{P}^{\prime}$, called the dual sequence of $\{P_{n}\}_{n\geq 0}$, such that,
\begin{equation}\label{SucDual}
\langle u_{n},P_{m}\rangle=\delta_{n,m}, \quad n,m\geq 0.
\end{equation}
We say that a sequence of forms $\{v_n \}_{ n \geq 0}$ is normalised if and only if $\left( v_n  \right)_n = 1, $ $n \geq 0$, and if $n \geq 1$, then $\left( v_n  \right)_m = 0, $ $m=0, \ldots, n-1$. Thus, the dual sequence  $\{u_n\}_{n\geq 0}$ is normalised. The first form $u_0$ is called the canonical form of ${\{P_{n}\}}_{n \geq 0}$.

We now introduce some operations on $\mathcal{P'}$ following \cite{Maroni-1991}.  
For $c \in \mathbb{C},~ f,p \in \mathcal{P}$, and $u \in \mathcal{P'}$, we define
\begin{align*}
&\langle fu, p \rangle = \langle u, fp \rangle, \quad
\langle u', p \rangle = -\langle u, p' \rangle,  \\
&\langle (x-c)^{-1}u, p \rangle = \langle u, \theta_c p \rangle
= \left\langle u, \frac{p(x)-p(c)}{x-c} \right\rangle.
\end{align*}
Given $f \in \mathcal{P}$ and $u \in \mathcal{P'}$, the product $uf$ is defined by  $(u f)(x) := \left\langle u, \displaystyle\frac{x f(x)-\zeta f(\zeta)}{x-\zeta} \right\rangle $.

This definition allows us to introduce the \emph{Cauchy product} of two forms $u,v \in \mathcal{P'}$ by 
\[
\langle uv, f \rangle := \langle u, v f \rangle, \quad f \in \mathcal{P}.
\]

In addition, we make use of the formal Stieltjes function associated with 
$u \in \mathcal{P}'$, defined by  \cite{Maroni-1991}
\[
S(u)(z) := - \sum_{n \geq 0} \frac{(u)_n}{z^{n+1}},
\]
which provides an alternative representation of the moment sequence $\{(u)_n\}_{n \geq 0}$.  
Since the moments uniquely determine $u$, the function $S(u)(z)$ does so as well.

\medskip

A linear functional $u$ is called \emph{regular} (or \emph{quasi-definite}) if there exists a sequence of polynomials  $\{P_n\}_{n \geq 0}$ such that \cite{Chihara-1978}
\begin{equation*}
\langle u, P_n P_m \rangle = r_n \, \delta_{n,m}, \quad n, m \geq 0,
\end{equation*}
where $\{r_n\}_{n \geq 0}$ is a sequence of nonzero complex numbers and $\delta_{n,m}$ denotes the Kronecker symbol.
The sequence $\{P_n\}_{n\geq 0}$ is then said orthogonal with respect to $u$. 
Then, necessarily, $\{P_n\}_{n \geq 0}$ is a PS, $u=(u)_0u_0$, and $\{P_n\}_{n \geq 0}$ and $u$ can be normalized. In the sequel, we shall consider that each $P_n(x)$ is monic, and $(u)_0=1$ (i.e. $u=u_0$).
Henceforth, a monic orthogonal
polynomial sequence $\{P_n\}_{n\geq 0}$ will be indicated as MOPS. 

It is well known that an MOPS is characterized by the following second-order linear recurrence relation and initial conditions, respectively \cite{Chihara-1978}
\begin{eqnarray}
&& P_0(x)=1\  ,\quad  P_1(x)=x-\beta_0, \label{ic_TTRR}\\
&& P_{n+2}(x)=(x-\beta_{n+1})P_{n+1}(x)-\gamma_{n+1}P_{n}(x),~~n\geq 0, \label{TTRR}
\end{eqnarray}
being $\{\beta_n\}_{n\geq 0}$ and $\{\gamma_{n+1}\}_{n\geq 0}$ sequences of complex numbers such that $\gamma_{n+1}\neq0$ for all $ n\geq 0$.

Let  $\{P_n^{(1)}\}_{n\geq0}$ be the associated polynomial sequence of order one of the MPS  $\{ P_n\}_{n\geq0}$ with respect to the canonical form $u=u_0$. It is well known that \cite{Chihara-1978}
$$
 P_n^{(1)}(x):=(u\theta_0P_{n+1})(x)=\bigg\langle  u,\frac{P_{n+1}(x)- P_{n+1}(\xi)}{x-\xi}  \bigg\rangle.
$$
The Stieltjes function of $u^{(1)}$ is expressed in terms of that of $u$ as  \cite{Maroni-1991} 
$$
\gamma_1 S\left(u^{(1)}\right)(z)=-\frac{1}{S(u)(z)}-\left(z-\beta_0\right).
$$
More generally, the sequence of associated polynomials of order $(r+1)$, $r\geq 1$, is  defined by recursion
$$
P_n^{(r+1)}(x)=\big(  P_n^{(r)} \big)^{(1)}(x),\quad  u_n^{(r+1)}=\big(  u_n^{(r)} \big)^{(1)},\quad n, r \geq 0.
$$

If $\{ P_n\}_{n\geq0}$ is a MOPS with respect to the form $u$, then for $r\in\mathbb{N}$, the associated sequence of polynomials of order $r$,  $\left\{ P_n^{(r)}\right\}_{n\geq0}$, is also orthogonal with respect to the form $u^{(r)}$ and satisfies the following recurrence relation 
\begin{eqnarray}
&&P_0^{(r)}(x)=1, \quad P_1^{(r)}(x)=x-\beta_0^{(r)},\label{ic_ASSTTRR}\\
&&P_{n+2}^{(r)}(x)=(x-\beta_{n+1}^{(r)})P_{n+1}^{(r)}(x)-\gamma_{n+1}^{(r)}P_n^{(r)}(x),\quad n\geq0,  \label{ASSTTRR}
\end{eqnarray}
where 
\begin{equation*}
\beta_{n}^{(r)}=\beta_{n+r}, \quad \gamma_{n+1}^{(r)}=\gamma_{n+1+r},\quad n\geq0.  
\end{equation*}

We recall the definition of the $r$-perturbed sequence $\{\widetilde{P}_n\}_{n\geq 0}$, for a fixed integer $r\geq 0$, associated with a MOPS $\{P_n\}_{n\geq 0}$, as introduced in \cite{Maroni-1991}. It is an MOPS satisfying the following second-order recurrence relation 
\[
\begin{array}{l}
\widetilde{P}_0(x)=1,\quad \widetilde{P}_1(x)=x-\widetilde{\beta}_0,\\[2pt]
\widetilde{P}_{n+2}(x)=(x-\widetilde{\beta}_{n+1})\widetilde{P}_{n+1}(x)-\widetilde{\gamma}_{n+1}\widetilde{P}_n(x),\quad n\geq0,
\end{array}
\]
with
\[
\begin{aligned}
&\widetilde{\beta}_0 = \beta_0 + \mu_0,\\[2pt]
&\widetilde{\beta}_n = \beta_n + \mu_n,\quad \mu_n\in\mathbb{C};\qquad 
\widetilde{\gamma}_n = \lambda_n\gamma_n,\quad \widetilde{\gamma}_n\in\mathbb{C}\setminus\{0\},\quad 1\leq n\leq r,\\[2pt]
&\widetilde{\beta}_n = \beta_n,\quad \widetilde{\gamma}_n = \gamma_n,\quad n\geq r+1.
\end{aligned}
\]
We assume that either $\mu_r\neq 0$ or $\lambda_n\neq 1$. The so-called co-recursive case corresponds to a perturbed case of order $0$. Using the notations $\mu:=(\mu_1,\ldots,\mu_r)$, $\lambda:=(\lambda_1,\ldots,\lambda_r)$, $r\geq 1$, we write
\[
\widetilde{P}_n(x)=P_n\left(\mu_0;\,{\mu\atop\lambda}\,;r;x\right),\qquad n\geq0,
\]
and the sequence $\{\widetilde{P}_n\}_{n\geq0}$ is orthogonal with respect to the perturbed form 
$$\widetilde{u}:=u\left(\mu_0;\,{\mu\atop\lambda}\,;r\right).$$

%%%%%%%%%%%%%%%%%%%%%%%%%%%%%%%%%%%
\subsection{Laguerre-Hahn forms}
%%%%%%%%%%%%%%%%%%%%%%%%%%%%%%%%%%%

\begin{definition} \cite{Dzoumba-these-1985,Magnus-1983,Maroni-1983}
A regular form $u$, with $(u)_0=1$, is said to be a Laguerre-Hahn form if its formal Stieltjes function satisfies the Riccati equation
\begin{equation}\label{Riccati}
A(z)S'(u)(z)=B(z)S^2(u)(z)+C(z)S(u)(z)+D(z),
\end{equation}
where $A$, $B$, $C$, and $D$ are polynomials.\\
The sequence \(\{P_n\}_{n\geq 0}\) orthogonal with respect to \( u \) is also called a Laguerre-Hahn sequence. 
\end{definition}

\begin{remark} \cite{Maroni-1991}
If $A=0$ identically, the form $u$ is classified as a second-degree form. 
If $A$ is not identically zero, it may be assumed, without loss of generality, 
that it is monic; and we let $A:=\Phi$. Under this normalization, the condition $B \neq 0$ 
characterizes $u$ as a strict Laguerre-Hahn form, whereas the case $B = 0$ 
corresponds to a semiclassical form.
\end{remark}

There are several characterizations of Laguerre-Hahn forms. Some of them are listed in the following result.
\begin{proposition} \cite{Alaya-Maroni,Bouakkaz-Maroni-1991,Dini-these-1988,Maroni-1991} 
Let $u$ be a regular and normalized form, i.e.,
 $(u)_0=1$, and let $\{P_n\}_{n\geq 0}$ be its corresponding MOPS. The following statements are equivalent
\begin{enumerate}
\item[(i)] $u$ is a Laguerre-Hahn form satisfying \eqref{Riccati} with $A=\Phi$.
\item[(ii)]  \cite{Dini-these-1988} $u$ satisfies the functional equation
\begin{equation}\label{Laguerre-Hahn-EF}
(\Phi u)'+\psi u+B(x^{-1}u^2)=0,
\end{equation}
where $\Phi$, $B$, $C$, and $D$ are the polynomials in \eqref{Riccati} and
\begin{align*}
&C=-\Phi'-\psi,\\
&D=-(u\theta_0\Phi)'-(u\theta_0\psi)-(u^2\theta_0^2B).
\end{align*}
\item[(iii)] \cite{Dini-these-1988} Each polynomial \(P_n, n \geq 0\), verifies the so-called structure relation
\begin{equation}\label{Dini_St_Rel}
\Phi(x)P_{n+1}^{\prime}(x) - B(x)P_n^{(1)}(x) = \sum_{\mu=n-s}^{n+d} \theta_{n,\mu} P_\mu(x), \quad n \geq s + 1,
\end{equation}
where \(\Phi\) and \(B\) are the polynomials defined in (i),  \(\{P_n^{(1)}\}_{n \geq 0}\) is the sequence of associated orthogonal polynomials of order 1 of \(\{P_n\}_{n \geq 0}\),  
$d=\max(t,q)$, $s=\max(p-1,d-2)$, being $t$, $p$, and $q$ the degrees of  $\Phi$, $\psi$, and $B$, respectively. 
\end{enumerate}
\end{proposition}

%%%%%%%%%%%%%%%%%%%%%%%%%%%%%%%%%%
It is worth noting that the above functional equation \eqref{Laguerre-Hahn-EF} is not uniquely determined.
Indeed, if $u$ is a Laguerre-Hahn form and $\chi$ is an arbitrary polynomial, then $u$ also satisfies
\[
(\chi \Phi u)' + \bigl(\chi \psi - \chi' \Phi\bigr) u
+ (\chi B)\,(x^{-1}u^{2}) = 0.
\]
This observation motivates the following definition.
\begin{definition} \cite{Alaya-Maroni, Bouakkaz-Maroni-1991}
The class of a Laguerre-Hahn form $u$ is the non-negative integer defined as
$$
s:=\min\max\big\{\deg{\psi}-1, \max\{\deg{\Phi}, \deg{B}\}-2\big\},
$$
where the minimum is taken among all polynomials $\Phi, \psi$ and $B$ such that $u$
satisfies \eqref{Laguerre-Hahn-EF}.
\end{definition}

Taking into account that the class of a Laguerre-Hahn form is crucial to state a hierarchy of such families, we need to give a criterion to characterize it.
\begin{proposition}\label{proposition-simplification} \cite{Alaya-Maroni,Bouakkaz-Maroni-1991}
Let $u$ be a Laguerre-Hahn form and let $\Phi$ and $\psi$ be non-zero polynomials 
such that \eqref{Laguerre-Hahn-EF} holds.
 Let
 \begin{equation}\label{s_class_LH}
 s=\max\big\{\deg{\psi}-1, \max\{\deg{\Phi}, \deg{B}\}-2\big\}.
 \end{equation} 
 Then $s$ is the class of $u$ if and only if
\begin{equation*}
\prod_{c\in\mathcal{Z}_{\Phi}}{\Big(|\Phi'(c)+\psi(c)|+|B(c)|+|\langle u, \theta_c^2\Phi+\theta_c\psi+u\theta_0\theta_c B\rangle|\Big)}\neq0,
\end{equation*}
where $\mathcal{Z}_{\Phi}$ denotes the set of zeros of $\Phi$.
\end{proposition}
\begin{remark}
When it is possible to simplify by the factor \( x - c \), we obtain the new functional equation
\[ 
((\theta_c\Phi)u)' + (\theta_c\psi + \theta_c^2\Phi)u + (\theta_cB)(x^{-1}u^2) = 0. 
\]
Then \( u \) is of class less than or equal to \( s - 1 \).    
\end{remark}

Based on Proposition~\ref{proposition-simplification}, one obtains an alternative criterion to determine the class using the polynomials involved in the Stieltjes equation of Riccati type~\eqref{Riccati}.
\begin{corollary} \cite{Alaya-Maroni}
Let $u$ be a Laguerre-Hahn form and let $A=\Phi$, $B$, $C$, and $D$ be non-zero polynomials satisfying \eqref{Riccati}. Then $s$ given by \eqref{s_class_LH} is the class of $u$ if and only if the polynomials $\Phi$, $B$, $C$, and $D$ are coprime or, equivalently,
$$
\prod_{c\in\mathcal{Z}_{\Phi}}{\big(|B(c)|+|C(c)|+|D(c)|\big)}\neq0.
$$
\end{corollary}

Let $a \in \mathbb{C}\setminus\{0\}$ and $b \in \mathbb{C}$. 
If a Laguerre-Hahn form $u$ of class $s$ satisfies \eqref{Laguerre-Hahn-EF}, 
then the shifted form $\tilde{u} = (h_{a^{-1}} \circ \tau_{-b})u$ 
is also a Laguerre-Hahn form of class $s$ and satisfies \cite{Bouakkaz-these,Bouakkaz-Maroni-1991,Dini-these-1988}
\begin{equation*}
(\tilde{\Phi}\,\tilde{u})' + \tilde{\psi}\,\tilde{u}
+ \tilde{B}\bigl(x^{-1}\tilde{u}^2\bigr) = 0,
\end{equation*}
where 
$$\tilde{\Phi}(x) = a^{-\deg\Phi}\,\Phi(ax+b),\ 
\tilde{\psi}(x) = a^{1-\deg\Phi}\,\psi(ax+b),\ 
\tilde{B}(x) = a^{-\deg\Phi}B(ax+b).$$

Consequently, a shifting transformation preserves both the Laguerre-Hahn character and
the class of the form. As a result, one may work with canonical
functional equations by appropriately relocating the zeros of $\Phi$ in
\eqref{Laguerre-Hahn-EF}.\\

%%%%%%%%%%%%%%%%%%%%%%%%%%%%%%%%%%%
\begin{proposition}\label{proposition-FSR} \cite{Dini-these-1988,Dzoumba-these-1985,Maroni-1991}
Let \(\{P_n\}_{n \geq 0}\) be a MOPS with respect to \(u\), satisfying \eqref{TTRR}-\eqref{ic_TTRR}. The following statements are equivalent.
\begin{enumerate}
\item[(i)] \(u\) is a Laguerre-Hahn form of class \(s\) satisfying \eqref{Laguerre-Hahn-EF}.

\item[(ii)] \(\{P_n\}_{n \geq 0}\) satisfies the following structure relation 
\begin{align}
\Phi(x) P_{n+1}'(x) - B_0(x) P_n^{(1)} (x) =& \frac{1}{2}\big(C_{n+1}(x) - C_0(x)\big) P_{n+1} (x)\nonumber\\ 
&- \gamma_{n+1} D_{n+1} (x) P_n (x), \quad n \geq 0, \label{R4}    
\end{align}
where \(C_n\) and \(D_n\) are polynomials with coefficients depending on \(n\), such that 
$\deg C_n \leq s + 1, ~ \deg D_n \leq s,$ 
satisfying the recurrence relations
\begin{align}
C_{n+1}(x) =& -C_n(x) + 2(x - \beta_n) D_n(x), \label{SR-1}\\
\gamma_{n+1} D_{n+1} (x) =& -\Phi(x) + \gamma_n D_{n-1} (x) - (x - \beta_n) C_n (x) + (x - \beta_n)^2 D_n (x), \label{SR-2}
\end{align}
for every \(n \geq 0\), with the initial conditions 
\begin{align}
B_0(x)=&D_{-1}(x)=B(x), \nonumber\\
C_0(x)=&C(x)=-\Phi'(x)-\psi(x), \label{C0-def}\\
D_0(x)=&D(x)=-(u\theta_0\Phi)'(x)-(u\theta_0\psi)(x)-(u^2\theta_0^2B)(x). \label{D0-def}
\end{align}
\end{enumerate}
\end{proposition} 

In the sequel, we will need the following lemma.
\begin{lemma} \cite{Alaya-these-1996,Bouakkaz-these}\label{Lemma-M=N=0}
Let $\left\{B_n\right\}_{n \geq 0}$ be a MOPS, and let $M(x ; n)$, and $N(x ; n)$ be two polynomials such that
$$
M(x ; n) B_{n+1}(x)=N(x ; n) B_n(x), \quad n \geq 0.
$$
If $\operatorname{deg} N(x ; n) \leq n$, then $M(x ; n)=0$ and $N(x ; n)=0$.
\end{lemma} 

%%%%%%%%%%%%%%%%%%%%%%%%%%%%%%%%%
\section{The coefficients of the main Laguerre-Hahn structure relation}\label{Section3}
%%%%%%%%%%%%%%%%%%%%%%%%%%%%%%%%%

In this section, we first recall the deduction of the nonlinear system governing the recurrence coefficients \(\beta_n\) and \(\gamma_{n+1}\) for a Laguerre–Hahn sequence of class zero, as established in \cite{Bouakkaz-these}. Subsequently, based on this system, we establish a main result that provides the coefficients of the main structure relation.  That result will be used in Section~\ref{Section4} to make explicit those coefficients for all Laguerre–Hahn forms of class zero.

In the sequel, we assume that $\left\{P_n\right\}_{n \geq 0}$ is a Laguerre-Hahn sequence of class zero satisfying \eqref{ic_TTRR}-\eqref{TTRR} and its corresponding form $u$ satisfying \eqref{Laguerre-Hahn-EF} with
\[ 
\Phi(x) = c_2 x^2 + c_1 x + c_0, \quad \psi(x) = a_1 x + a_0, \quad B(x) = b_2 x^2 + b_1 x + b_0,
\]
with 
\[
|c_2| + |a_1| + |b_2| > 0.
\]

%%%%%%%%%%%%%%%%%%%%%%%%%%%%%%%%%%%
\subsection{System fulfilled by the recurrence coefficients \(\beta_n\), \(\gamma_{n+1}\)}
%%%%%%%%%%%%%%%%%%%%%%%%%%%%%%%%%%%
The non-linear system satisfied by $\beta_n$ and $\gamma_{n+1}$ for a Laguerre-Hahn sequence of class zero is established in \cite{Bouakkaz-these} via the functional equation satisfied by its corresponding linear form $u$. In fact, with \eqref{TTRR}, the relation \eqref{Dini_St_Rel} becomes
\begin{equation}\label{(5.1.1)-Maroni}
\Phi(x) P'_{n+1}(x) - B(x) P^{(1)}_n(x) = (G_n x + E_n) P_{n+1}(x) + F_n P_n(x), \quad n \geq 0.
\end{equation}
Comparing the highest degree terms in the structure relation \eqref{(5.1.1)-Maroni} yields 
\[
G_n = c_2(n+1) - b_2, \quad n \geq 0.
\]
Differentiating \eqref{TTRR} and multiplying the resulting relation by $\Phi$ gives
\[
\Phi P'_{n+2} = (x-\beta_{n+1})\Phi P'_{n+1} + \Phi P_{n+1} - \gamma_{n+1}\Phi P'_n,\quad n\geq 0. 
\]
Substituting the structure relation \eqref{(5.1.1)-Maroni} for the indices $n+1$, $n$, and $n-1$ into the last relation gives
\[
\begin{aligned}
&B P_{n+1}^{(1)} + (G_{n+1}x+E_{n+1})P_{n+2} + F_{n+1}P_{n+1}\\
&= (x-\beta_{n+1})\bigl[B P_n^{(1)} + (G_n x+E_n)P_{n+1} + F_n P_n\bigr]\\
&\quad + \Phi P_{n+1} - \gamma_{n+1}\bigl[B P_{n-1}^{(1)} + (G_{n-1}x+E_{n-1})P_n + F_{n-1}P_{n-1}\bigr].
\end{aligned}
\]
Using \eqref{ASSTTRR} with $r=1$, the combination $B P_{n+1}^{(1)} - (x-\beta_{n+1})B P_n^{(1)} + \gamma_{n+1}B P_{n-1}^{(1)}$ vanishes. Hence all terms containing $B$ cancel, leaving
\begin{align}
(G_{n+1}x+&E_{n+1})P_{n+2} + F_{n+1}P_{n+1} \nonumber\\
 =& (x-\beta_{n+1})\bigl[(G_n x+E_n)P_{n+1} + F_n P_n\bigr]\nonumber\\
& + \Phi P_{n+1} - \gamma_{n+1}\bigl[(G_{n-1}x+E_{n-1})P_n + F_{n-1}P_{n-1}\bigr]. \label{Eq-9}
\end{align}
On the other hand, from \eqref{TTRR} we can express
\(
P_{n+2} = (x-\beta_{n+1})P_{n+1} - \gamma_{n+1}P_n\) and \(P_{n-1} = {\gamma_n}^{-1}\bigl[(x-\beta_n)P_n - P_{n+1}\bigr]
\). Then, substituting into \eqref{Eq-9} and collecting terms in $P_{n+1}$ and $P_n$ yields an equation of the form
\[
M(x,n)P_{n+1}(x) = N(x,n)P_n(x), \quad n\geq 1,
\]
with
\[
\begin{aligned}
M(x,n) &= \bigl[c_1 - E_{n+1} + \beta_{n+1}G_{n+1} + E_n - \beta_{n+1}G_n\bigr]x\\
&\quad + \Bigl[c_0 + \beta_{n+1}E_{n+1} + \frac{\gamma_{n+1}}{\gamma_n}F_{n-1} - \beta_{n+1}E_n - F_{n+1}\Bigr],
\end{aligned}
\]
and
\[
\begin{aligned}
N(x,n) &= \bigl[\gamma_{n+1}G_{n-1} + \frac{\gamma_{n+1}}{\gamma_n}F_{n-1} - F_n - \gamma_{n+1}G_{n+1}\bigr]x\\
&\quad + \Bigl[\gamma_{n+1}E_{n-1} + \beta_{n+1}F_n - \gamma_{n+1}E_{n+1} - \beta_n\frac{\gamma_{n+1}}{\gamma_n}F_{n-1}\Bigr].
\end{aligned}
\]
For $n \geq 1$, $M(x,n)$ and $N(x,n)$ are polynomials of degree at most one. By Lemma~\ref{Lemma-M=N=0}, they vanish identically.\\
The vanishing of both coefficients of $M(x,n)$ and $N(x,n)$ for $n\geq 1$, together with the structure relation for $n=0$, yields that the coefficients $\beta_n$, $\gamma_{n+1}$, $E_n$, $F_n$, $n \geq 0$ satisfy the following system \cite[Lemme~2.2]{Bouakkaz-these}:
\begin{align*}
&E_0 = c_1 - b_1 + \beta_0 (c_2 - b_2), \\
&F_0 = \Phi(\beta_0) - B(\beta_0), \\
&E_{n+1} - E_n - c_2 \beta_{n+1} - c_1 = 0, \quad n \geq 1, \\
&\gamma^{-1}_n F_{n-1} - \gamma^{-1}_{n+1} F_{n+1} + \beta_{n+1} \gamma^{-1}_{n+1} (E_{n+1} - E_n) + \gamma^{-1}_{n+1} c_0 = 0, \quad n \geq 1, \\
&\gamma^{-1}_{n+1} F_n - \gamma^{-1}_{n} F_{n-1} + 2c_2 = 0, \quad n \geq 1, \\
&E_{n-1} - E_{n+1} + \beta_{n+1} \gamma^{-1}_{n+1} F_n - \beta_n \gamma^{-1}_{n} F_{n-1} = 0, \quad n \geq 1. 
\end{align*} 
This system provides a set of algebraic relations linking $\beta_n$, $\gamma_{n+1}$, $E_n$ and $F_n$. By introducing the auxiliary quantities $r_n = F_n / \gamma_{n+1}$ and eliminating $E_n$ and $F_n$ through successive substitutions, we obtain, after straightforward calculations, the following system \cite[Lemme~2.3]{Bouakkaz-these}:
\begin{align}
&r_0 = (\Phi(\beta_0) - B(\beta_0)) \gamma^{-1}_1, \label{r0}\\
&\{(2n+1)c_2 - r_0\} \beta_{n+1} - \{(2n-3)c_2 - r_0\} \beta_n + 2c_1 = 0, \quad n \geq 1, \label{(3-34)Houria-System} \\
&(c_2 - r_0)\beta_1 - (-3c_2 - r_0 + 2b_2)\beta_0 + 2c_1 - b_1 = 0,  \label{(3-35)Houria-System} \\
&\{2(n+1)c_2 - r_0\}\gamma_{n+2} - \{2(n-1)c_2 - r_0\}\gamma_{n+1} + \Phi(\beta_{n+1}) = 0, \quad n \geq 1, \label{nnb1-Houria-System}\\
&(2c_2 - r_0)\gamma_2 - (-2c_2 - r_0 + b_2)\gamma_1 + \Phi(\beta_1) = 0, \label{nnb2-Houria-System} \\
&E_n = c_2 \sum_{\nu=0}^n \beta_\nu + (n+1)c_1 - (b_1 + \beta_0 b_2), \quad n \geq 0, \label{En-System}  \\
&F_n = (r_0 - 2c_2 n)\gamma_{n+1}, \quad n \geq 0. \label{Fn-System}
\end{align}

%%%%%%%%%%%%%%%%%%%%%%%%%%%%%%%%%
\subsection{Main structure relation}
%%%%%%%%%%%%%%%%%%%%%%%%%%%%%%%%%

The following result provides explicit formulas for the coefficients in the structure relation of any normalized 
 Laguerre-Hahn form of class zero. These expressions depend only on the coefficients of the polynomials $\Phi$, $\psi$, and $B$ appearing in the functional equation, as well as on the recurrence coefficients $\beta_n$.
\begin{theorem}\label{theorem-SR}
Let $\{P_n\}_{n\geq 0}$ be a MOPS with respect to the normalized Laguerre–Hahn form $u$ of class $s=0$, fulfilling the structure relation~\eqref{R4}. Then we have
\begin{align}
r_0 =&a_1+2b_2-c_2, \label{r0-explicite}\\
C_0(x) =& -(2c_2 + a_1)x - (c_1 + a_0), \label{C0-explicite}\\
D_0(x) =& -(c_2 + b_2 + a_1), \label{D0-explicite}\\
C_{n+1}(x) =& \big[2n c_2 - (a_1 + 2b_2)\big]x +  \beta_{n+1}\big[(a_1 + 2b_2) - 2(n+1)c_2\big] - c_1, \quad n\geq0, \label{Cn+1-explicite}\\
D_{n+1}(x) =& (2n+1)c_2 - (a_1 + 2b_2),\quad n\geq0. \label{Dn+1-explicite}
\end{align}
\end{theorem}
\begin{proof}
To prove relation~\eqref{r0-explicite}, we start from the functional equation satisfied by the moments of orders $n=0$ and $n=1$, that is,
\begin{equation}\label{Mom-0-1}
\bigl((\Phi u)' + \psi u + B(x^{-1}u^2)\bigr)_0 = 0
\quad \text{and} \quad
\bigl((\Phi u)' + \psi u + B(x^{-1}u^2)\bigr)_1 = 0,
\end{equation}
and since $(u)_0=1$, $(u)_1=\beta_0$, and $(u)_2=\beta_0^2+\gamma_1$, we obtain the following 
\begin{equation*}
\bigl((\Phi u)'\bigr)_0 = 0, \quad
(\psi u)_0 = \psi(\beta_0), \quad
\bigl(B(x^{-1}u^2)\bigr)_0 = 2b_2\beta_0 + b_1 = B'(\beta_0),
\end{equation*}
and
\begin{align*}
\bigl((\Phi u)'\bigr)_1
&= -\bigl[c_2 (u)_2 + c_1 (u)_1 + c_0 (u)_0\bigr]
= -\Phi(\beta_0) - c_2\gamma_1, \\
(\psi u)_1
&= a_1 (u)_2 + a_0 (u)_1
= \beta_0\,\psi(\beta_0) + a_1\gamma_1, \\
\bigl(B(x^{-1}u^2)\bigr)_1
&= 2b_2 (u)_0 (u)_2 + b_2 (u)_1^2
+ 2b_1 (u)_0 (u)_1 + b_0 (u)_0^2 \\
&= B(\beta_0) + \beta_0 B'(\beta_0) + 2b_2\gamma_1.
\end{align*}
Consequently, the relations in~\eqref{Mom-0-1} become
\begin{equation}\label{phi+B'}
\psi(\beta_0) + B'(\beta_0) = 0,    
\end{equation}
and
\[
-\Phi(\beta_0) - c_2\gamma_1
+ \beta_0 \, \psi(\beta_0) + a_1\gamma_1
+ B(\beta_0) + \beta_0 \, B'(\beta_0) + 2b_2\gamma_1 = 0.
\]
This yields
\[
\Phi(\beta_0) - B(\beta_0)
= \gamma_1\bigl(a_1 + 2b_2 - c_2\bigr).
\]
Therefore, relation~\eqref{r0-explicite} follows by comparing the above identity with~\eqref{r0} and by using the fact that $\gamma_1\neq 0$, which is ensured by the regularity of $u$.

Relation~\eqref{C0-explicite} is readily obtained from~\eqref{C0-def}.
On the other hand, by a direct computation, relation~\eqref{C0-def} yields
\[
D_0(x) = -(u)_0 (c_2 + a_1) - b_2 (u)_0^2,
\]
which coincides exactly with~\eqref{D0-explicite} since the form $u$ is normalized.

Now, comparing the two structural relations \eqref{R4} and \eqref{(5.1.1)-Maroni}, we have 
\[
\left(\frac{1}{2}\left(C_{n+1}(x)-C_0(x)\right)-(G_n x+E_n)\right) P_{n+1}(x)=\left(\gamma_{n+1} D_{n+1}(x)+F_n\right) P_n(x), \quad n\geq 0,
\]
By orthogonality, and using Lemma~\ref{Lemma-M=N=0}, we get
\begin{align}
\frac12\big(C_{n+1}(x) - C_0(x)\big) &= G_n x + E_n, , \quad n\geq 0, \label{Cn+1C0-GnEn}\\
-\gamma_{n+1}D_{n+1}(x) &= F_n, \quad n\geq 0. \label{Dn+1-Fn}    
\end{align}
From~\eqref{Dn+1-Fn} and~\eqref{Fn-System}, we obtain
\[
D_{n+1}(x) = 2c_2 n - r_0, \quad n\geq 0,
\]
which is exactly the desired relation~\eqref{Dn+1-explicite} upon substituting $r_0$ by its explicit expression given in~\eqref{r0-explicite}.
It therefore remains to establish relation~\eqref{Cn+1-explicite}. To this end, we first rewrite~\eqref{Cn+1C0-GnEn} as
\[
C_{n+1}(x) = 2\bigl(G_n x + E_n\bigr) + C_0(x), \quad n\geq 0.
\]
Using \(G_n = c_2(n+1)-b_2\) and \(C_0(x)=-(2c_2+a_1)x-(c_1+a_0)\), we compute
\[
2G_n x + C_0(x) = \bigl(2c_2n - (a_1+2b_2)\bigr)x - c_1 - a_0, , \quad n\geq 0,
\]
hence
\begin{equation}\label{(1)}
C_{n+1}(x) = (2c_2n - (a_1+2b_2))x + 2E_n - c_1 - a_0, \quad n\geq 0. 
\end{equation}
Now, from \eqref{En-System},
\[
2E_n = 2c_2\sum_{\nu=0}^n \beta_\nu + 2(n+1)c_1 - 2b_1 - 2\beta_0 b_2, \quad n\geq 0.
\]
Denote \(\Sigma_n = \sum_{\nu=0}^n \beta_\nu\).  
Using the recurrence \eqref{(3-34)Houria-System} together with \(r_0 = (a_1+2b_2) - c_2\) we obtain the equivalent form
\[
\bigl[2(n+1)c_2 - (a_1+2b_2)\bigr]\beta_{n+1} - \bigl[2(n-1)c_2 - (a_1+2b_2)\bigr]\beta_n + 2c_1 = 0,\quad n\geq 1. 
\]
Define \(\varpi_k := [2(k-1)c_2 - (a_1+2b_2)]\beta_k\). Then the last relation becomes
\[
\varpi_{k+1} - \varpi_k + 2c_2\beta_{k+1} = -2c_1, \quad k\geq 1.
\]
Summing for \(k=1\) to \(n\) gives
\[
\varpi_{n+1} - \varpi_1 + 2c_2\sum_{k=1}^n \beta_{k+1} = -2nc_1.
\]
With \(\varpi_1 = -(a_1+2b_2)\beta_1\), \(\varpi_{n+1}=[2nc_2 - (a_1+2b_2)]\beta_{n+1}\) and \(\sum_{k=1}^n \beta_{k+1}=\Sigma_{n+1}-\beta_0-\beta_1\), we obtain after simplification
\begin{equation}\label{(2)}
2c_2 \Sigma_n = -[2(n+1)c_2 - (a_1+2b_2)]\beta_{n+1} - ((a_1+2b_2)-2c_2)\beta_1 + 2c_2\beta_0 - 2nc_1. 
\end{equation}
Substituting \eqref{(2)} into the expression of \(2E_n\) yields
\begin{align}
2E_n &= -[2(n+1)c_2 - (a_1+2b_2)]\beta_{n+1} - ((a_1+2b_2)-2c_2)\beta_1 \nonumber \\
&\quad + 2(c_2-b_2)\beta_0 + 2c_1 - 2b_1. \label{(3)}
\end{align}
From \eqref{phi+B'} we have \(b_1 = -a_0 - (a_1+2b_2)\beta_0\).  
Insert this into \eqref{(3)}, we get
\begin{align}
2E_n &= -[2(n+1)c_2 - (a_1+2b_2)]\beta_{n+1} - ((a_1+2b_2)-2c_2)\beta_1 \nonumber\\
&\quad + 2(c_2+a_1+b_2)\beta_0 + 2c_1 + 2a_0. \label{(4)}
\end{align}
We now establish a relation linking \(\beta_0\), \(\beta_1\), \(c_1\), \(a_0\) and the coefficients.  
Starting from \eqref{(3-35)Houria-System} and using \(r_0 = a_1+2b_2-c_2\) from \eqref{r0-explicite}, we have
\begin{equation}\label{(5)}
(2c_2 - (a_1+2b_2))\beta_1 + (2c_2 + a_1)\beta_0 + 2c_1 - b_1 = 0. 
\end{equation}
Now eliminate \(b_1\) using \eqref{phi+B'}. Substituting into \eqref{(5)} yields
\[
(2c_2 - (a_1+2b_2))\beta_1 + (2c_2 + a_1)\beta_0 + 2c_1 + a_0 + (a_1+2b_2)\beta_0 = 0.
\]

This implies
\begin{equation}\label{(6)}
(2c_2 - (a_1+2b_2))\beta_1 + 2(c_2 + a_1 + b_2)\beta_0 + 2c_1 + a_0 = 0. 
\end{equation}
Returning to \eqref{(4)}, we note that the combination
\[
-((a_1+2b_2)-2c_2)\beta_1 + 2(c_2 + a_1 + b_2)\beta_0
\]
is exactly the left-hand side of \eqref{(6)} without the terms \(2c_1 + a_0\). Indeed,
\[
-((a_1+2b_2)-2c_2)\beta_1 + 2(c_2 + a_1 + b_2)\beta_0 = (2c_2 - (a_1+2b_2))\beta_1 + 2(c_2 + a_1 + b_2)\beta_0.
\]
Therefore, by \eqref{(6)}, this combination equals \(-2c_1 - a_0\). Substituting into \eqref{(4)} gives
\[
2E_n = -[2(n+1)c_2 - (a_1+2b_2)]\beta_{n+1} + (-2c_1 - a_0) + 2c_1 + 2a_0,
\]
which simplifies to
\begin{equation}\label{(7)}
2E_n = -[2(n+1)c_2 - (a_1+2b_2)]\beta_{n+1} + a_0. 
\end{equation}

Finally, insert \eqref{(7)} into \eqref{(1)}
\[
C_{n+1}(x) = (2c_2n - (a_1+2b_2))x + \bigl(-[2(n+1)c_2 - (a_1+2b_2)]\beta_{n+1} + a_0\bigr) - c_1 - a_0.
\]
which is precisely the formula \eqref{Cn+1-explicite}, thereby completing the proof.
\end{proof}

%%%%%%%%%%%%%%%%%%%%%%%%%%%%%%%%%
\section{Ten canonical cases}\label{Section4}
%%%%%%%%%%%%%%%%%%%%%%%%%%%%%%%%%

The purpose of this section is twofold. It is worth noting that the resolution of the system \eqref{r0}--\eqref{Fn-System} presented in \cite{Bouakkaz-these,Bouakkaz-Maroni-1991} is incomplete, and that the coefficients of the Laguerre-Hahn structure relation corresponding to the eight obtained families are not properly stated. Therefore, on the one hand, we provide a complete resolution of the system \eqref{r0}--\eqref{Fn-System}, which leads to two new families analogous to the Bessel case that have not been reported in the literature. On the other hand, we establish the coefficients of the Laguerre-Hahn structure relation for the ten resulting families, correcting in particular those that were presented in \cite{Bouakkaz-these}.

In fact, the resolution of \eqref{r0}--\eqref{Fn-System} leads to ten canonical cases (see the ten tables below), according to the different values of
\[
\Phi(x) =
\begin{cases}
1 & , \quad (\text{two analogous to Hermite}), \\
x & , \quad (\text{two analogous to Laguerre}), \\
x^2 & , \quad (\text{four analogous to Bessel}), \\
x^2 - 1 & , \quad (\text{two analogous to Jacobi}). 
\end{cases}
\]

\vspace{-0.2cm}
%%%%%%%%%%%%%%%%%%%%%%%%%%%%%%%%%
\subsection{Two cases analogous to Hermite case}
%%%%%%%%%%%%%%%%%%%%%%%%%%%%%%%%%

In this subsection, we present a complete description of Laguerre–Hahn forms of class zero, analogous to the classical Hermite case, in terms of the recurrence coefficients, the functional equation, and the coefficients of the Laguerre–Hahn main structure relation. These forms correspond to the canonical setting $\Phi(x)=1$, in which case the system \eqref{r0}--\eqref{Fn-System} yields the following result:
\begin{lemma}\cite{Bouakkaz-these}
If \(\Phi(x) = 1\), then the coefficients $\beta_n$ and $\gamma_n$ are given by
\begin{align*}
\beta_{n+1} &= \beta_0 - \frac{1}{r_0} (b_1+2b_2\beta_0), \quad n \geq 0, \\
\gamma_{n+1} &= \left(1 - \frac{b_2}{r_0 }\right) \gamma_1 + \frac{n}{r_0}, \quad n \geq 1. 
\end{align*}
\end{lemma}

In \cite{Bouakkaz-these,Bouakkaz-Maroni-1991}, the authors determined all strict Laguerre-Hahn forms of class $s=0$ analogous to the classical Hermite case. These are precisely two families:
\begin{enumerate}
\item[$\bullet$] Case 1: $b_2 \neq 2$.
\item[$\bullet$] Case 2: $b_2 = 2$.
\end{enumerate}

\vspace{-0.3cm}
%%%%%%%%%%%%%%%%%%%%%%%%%%%%%%%%%
\subsubsection{Case 1 analogous to Hermite}
%%%%%%%%%%%%%%%%%%%%%%%%%%%%%%%%%
\begin{table}[htbp]
\centering
\renewcommand{\arraystretch}{1.5}
\setlength{\tabcolsep}{1em} 

\begin{tabular}{@{}p{0.3\textwidth} p{0.50\textwidth}@{}}

\midrule
\begin{minipage}[t]{0.3\textwidth}
\raggedright \textbf{Regularity conditions}
\end{minipage}
&
\begin{minipage}[t]{0.50\textwidth}
\raggedright
$\lambda, \rho, \tau \in \mathbb{C},\quad \rho \neq 0,\quad \tau\neq -n, \quad n \geq 1$.
\end{minipage}
\\
\midrule

\begin{minipage}[t]{0.3\textwidth}
\raggedright \textbf{The recurrence coefficients}
\end{minipage}
&
\begin{minipage}[t]{0.50\textwidth}
\raggedright
$\beta_0 = \lambda$, \quad $\beta_{n+1} = 0$, \quad $n \geq 0$, \\
$\gamma_1 = \rho \displaystyle\frac{\tau+1}{2}$, \quad $\gamma_{n+1} = \displaystyle\frac{n+\tau+1}{2}$, \quad $n \geq 1$.
\end{minipage}
\\
\midrule

\begin{minipage}[t]{0.3\textwidth}
\raggedright \textbf{The functional equation}
\end{minipage}
&
\begin{minipage}[t]{0.50\textwidth}
\raggedright
$\Phi(x) = 1, \quad
\psi(x) = 2\displaystyle\frac{2-\rho}{\rho}x - \frac{4\lambda}{\rho}$,\\
$B(x) = 2\displaystyle\frac{\rho-1}{\rho}x^2 + 2\lambda \displaystyle\frac{2-\rho}{\rho}x + 1 - \rho(\tau+1) - \displaystyle\frac{2\lambda^2}{\rho}$.
\end{minipage}
\\
\midrule

\begin{minipage}[t]{0.3\textwidth}
\raggedright \textbf{The coefficients of the Laguerre-Hahn structure relation}
\end{minipage}
&
\begin{minipage}[t]{0.50\textwidth}
\raggedright
$C_0(x) = 2\displaystyle\frac{\rho-2}{\rho}x + \frac{4\lambda}{\rho}, \quad D_0(x) = -\displaystyle\frac{2}{\rho}$, \\
$C_{n+1}(x) = -2x, \quad  D_{n+1}(x) = -2, \quad n \geq 0$.
\end{minipage}
\\
\bottomrule
\end{tabular}
\caption{Case 1 analogous to Hermite}
\label{tab1-Hermite}
\end{table}
\FloatBarrier
\vspace*{-\baselineskip}
\vspace{-0.4cm}
When \(\tau \in \mathbb{N}\), the form \(u_0\) can be read as
\[
u_0 = \mathcal{H}^{(\tau)} \left(\lambda;\,{0\atop\rho}\,;1\right),
\]
where $\mathcal{H}$ represents the classical Hermite form.
%%%%%%%%%%%%%%%%%%%%%%%%%%%%%%%%%
\subsubsection{Case 2 analogous to Hermite}
%%%%%%%%%%%%%%%%%%%%%%%%%%%%%%%%%

\begin{table}[htbp]
\centering
\renewcommand{\arraystretch}{1.5}
\setlength{\tabcolsep}{1em} 

\begin{tabular}{@{}p{0.3\textwidth} p{0.60\textwidth}@{}}

\midrule
\begin{minipage}[t]{0.3\textwidth}
\raggedright \textbf{Regularity conditions}
\end{minipage}
&
\begin{minipage}[t]{0.40\textwidth}
\raggedright
$\lambda, \rho \in \mathbb{C},\quad \rho \neq 0$.
\end{minipage}
\\
\midrule

\begin{minipage}[t]{0.3\textwidth}
\raggedright \textbf{The recurrence coefficients}
\end{minipage}
&
\begin{minipage}[t]{0.6\textwidth}
\raggedright
$\beta_0 = \lambda$, ~~ $\beta_{n+1} = 0$, ~~ $n \geq 0; ~~
\gamma_1 = \displaystyle\frac{\rho}{2}$, ~~ $\gamma_{n+1} = \displaystyle\frac{n}{2}$, ~~ $n \geq 1$.
\end{minipage}
\\
\midrule

\begin{minipage}[t]{0.3\textwidth}
\raggedright \textbf{The functional equation}
\end{minipage}
&
\begin{minipage}[t]{0.60\textwidth}
\raggedright
$\Phi(x) = 1, \quad \psi(x) = -2x, \quad
B(x) = 2x^2-2\lambda x+1-\rho$.
\end{minipage}
\\
\midrule

\begin{minipage}[t]{0.3\textwidth}
\raggedright \textbf{The coefficients of the Laguerre-Hahn structure relation}
\end{minipage}
&
\begin{minipage}[t]{0.60\textwidth}
\raggedright
$C_0(x) = 2x, \quad 
C_{n+1}(x) = -2x, \quad n \geq 0 $, \\ 
$D_0(x) = 0, \quad D_{n+1}(x) = -2, \quad n \geq 0$.
\end{minipage}
\\
\bottomrule
\end{tabular}
\caption{Case 2 analogous to Hermite}
\label{tab2-Hermite}
\end{table}
\FloatBarrier
\vspace*{-\baselineskip}
\vspace{-0.2cm}
The form \(u_0\) can be read as
\[
u^{(1)}_0=\mathcal{H}. 
\]

%%%%%%%%%%%%%%%%%%%%%%%%%%%%%%%%%
\subsection{Two cases analogous to Laguerre case}
%%%%%%%%%%%%%%%%%%%%%%%%%%%%%%%%%
In this subsection, we present a complete characterization of Laguerre-Hahn forms of class zero analogous to the classical Laguerre case. These forms correspond to the canonical situation \(\Phi(x)=x\), in which case the system \eqref{r0}--\eqref{Fn-System} yields the following result:
\begin{lemma}\cite{Bouakkaz-these}
If \(\Phi(x) = x\), then the coefficients $\beta_n$ and $\gamma_n$ are given by
\begin{align*}
\beta_{n+1} &= \beta_0 + \frac{1}{r_0} \Big( 2(n+1)-b_1-2b_2\beta_0 \Big), \quad n \geq 0, \\
\gamma_{n+1} &= \left(1 - \frac{b_2}{r_0}\right) \gamma_1 + \frac{n}{r_0^2} \Big( n+1-b_1+(r_0-2b_2)\beta_0 \Big), \quad n \geq 1. 
\end{align*}
\end{lemma}

In \cite{Bouakkaz-these,Bouakkaz-Maroni-1991}, the authors determined all strict Laguerre-Hahn forms of class $s=0$ analogous to the classical Laguerre case. These are precisely two families:
\begin{enumerate}
\item[$\bullet$] Case 1: $b_2 \neq 1$.
\item[$\bullet$] Case 2: $b_2 = 1$.
\end{enumerate}

%%%%%%%%%%%%%%%%%%%%%%%%%%%%%%%%%
\subsubsection{Case 1 analogous to Laguerre}
%%%%%%%%%%%%%%%%%%%%%%%%%%%%%%%%%
\begin{table}[htbp]
\centering
\renewcommand{\arraystretch}{1.5}
\setlength{\tabcolsep}{1em} 

\begin{tabular}{@{}p{0.22\textwidth} p{0.65\textwidth}@{}}

\midrule
\begin{minipage}[t]{0.22\textwidth}
\raggedright \textbf{Regularity conditions}
\end{minipage}
&
\begin{minipage}[t]{0.65\textwidth}
\raggedright
$\lambda, \rho, \tau, \alpha\in\mathbb{C}, \quad \rho\neq 0, \quad \alpha+\tau\neq-(n+1),\quad n\geq 0$, \\
$\tau\neq-(n+1), \quad n\geq 0$.
\end{minipage}
\\
\midrule

\begin{minipage}[t]{0.22\textwidth}
\raggedright \textbf{The recurrence coefficients}
\end{minipage}
&
\begin{minipage}[t]{0.65\textwidth}
\raggedright
$\beta_0=2\tau+\alpha+1+\lambda, \quad 
\beta_{n+1}=2(n+\tau+1)+\alpha+1, \quad n\geq 0$,\\
$\gamma_1=\rho(1+\tau)(\tau+\alpha+1), \quad
\gamma_{n+1}=(n+\tau+1)(n+\tau+\alpha+1),\quad n\geq 1$.
\end{minipage}
\\
\midrule

\begin{minipage}[t]{0.22\textwidth}
\raggedright \textbf{The functional equation}
\end{minipage}
&
\begin{minipage}[t]{0.65\textwidth}
\raggedright
$\Phi(x)=x,\quad
\psi(x)=\displaystyle\frac{2-\rho}{\rho}x+\frac{\rho-2}{\rho}(2\tau+\alpha+1)-\frac{2\lambda}{\rho}$,\\
$B(x)=\displaystyle\frac{\rho-1}{\rho}x^2+\left\{2\displaystyle\frac{1-\rho}{\rho}(2\tau+\alpha+1)+\lambda\displaystyle\frac{2-\rho}{\rho}\right\} x$ \\
$\quad\quad\quad~ +(2\tau+\alpha+1+\lambda)\left\{\displaystyle\frac{\rho-1}{\rho}(2\tau+\alpha+1)+1-\frac{\lambda}{\rho}\right\}$\\
$\quad\quad\quad~ -\rho(\tau+1)(\tau+\alpha+1)$.
\end{minipage}
\\
\midrule

\begin{minipage}[t]{0.22\textwidth}
\raggedright \textbf{The coefficients of the Laguerre-Hahn structure relation}
\end{minipage}
&
\begin{minipage}[t]{0.65\textwidth}
\raggedright
$C_0(x)=\displaystyle\frac{\rho-2}{\rho}(x-2\tau-\alpha-1)+\frac{2\lambda}{\rho}-1,\quad D_{0}(x)=-\displaystyle\frac{1}{\rho}$,\\
$C_{n+1}(x)=-x+\alpha+2(n+\tau+1), \quad D_{n+1}(x)=-1, \quad n\geq 0$.
\end{minipage}
\\
\bottomrule
\end{tabular}
\caption{Case 1 analogous to Laguerre}
\label{tab1-Laguerre}
\end{table}
\FloatBarrier
\vspace*{-\baselineskip}
\vspace{-0.2cm}
When \(\tau \in \mathbb{N}\), the form \(u_0\), can be read as
\[ 
u_{0}=\mathcal{L}^{(\tau)}(\alpha) \left(\lambda;\,{0\atop\rho}\,;1\right),
\]
where $\mathcal{L}(\alpha)$ represents the classical Laguerre form.

%%%%%%%%%%%%%%%%%%%%%%%%%%%%%%%%%
\subsubsection{Case 2 analogous to Laguerre}
%%%%%%%%%%%%%%%%%%%%%%%%%%%%%%%%%
\begin{table}[htbp]
\centering
\renewcommand{\arraystretch}{1.5}
\setlength{\tabcolsep}{1em} 

\begin{tabular}{@{}p{0.22\textwidth} p{0.50\textwidth}@{}}

\midrule
\begin{minipage}[t]{0.22\textwidth}
\raggedright \textbf{Regularity conditions}
\end{minipage}
&
\begin{minipage}[t]{0.50\textwidth}
\raggedright
$\lambda, \rho, \alpha \in \mathbb{C},\quad \rho \neq 0,\quad \alpha \neq -n, \quad n \geq 1$.
\end{minipage}
\\
\midrule

\begin{minipage}[t]{0.22\textwidth}
\raggedright \textbf{The recurrence coefficients}
\end{minipage}
&
\begin{minipage}[t]{0.50\textwidth}
\raggedright
$\beta_0 = \alpha - 1 + \lambda$, \quad $\beta_{n+1} = 2n + \alpha + 1$, \quad $n \geq 0$, \\
$\gamma_1 = \rho$, \quad $\gamma_{n+1} = n(n + \alpha)$, \quad $n \geq 1$.
\end{minipage}
\\
\midrule

\begin{minipage}[t]{0.22\textwidth}
\raggedright \textbf{The functional equation}
\end{minipage}
&
\begin{minipage}[t]{0.50\textwidth}
\raggedright
$\Phi(x) = x, \quad
\psi(x) = -x + \alpha - 1$, \\
$B(x) = x^2 + \big\{2(1-\alpha) - \lambda\big\}x + \alpha(\alpha - 1 + \lambda) - \rho$.
\end{minipage}
\\
\midrule

\begin{minipage}[t]{0.22\textwidth}
\raggedright \textbf{The coefficients of the Laguerre-Hahn structure relation}
\end{minipage}
&
\begin{minipage}[t]{0.50\textwidth}
\raggedright
$C_0(x) = x - \alpha, \quad 
C_{n+1}(x) = -x + \alpha + 2n, \quad n \geq 0$, \\
$D_0(x) = 0, \quad D_{n+1}(x) = -1, \quad n \geq 0$.
\end{minipage}
\\
\bottomrule
\end{tabular}
\caption{Case 2 analogous to Laguerre}
\label{tab2-Laguerre}
\end{table}
\FloatBarrier
\vspace*{-\baselineskip}
\vspace{-0.4cm}
The form \(u_0\) can be read as
\[
u_0^{(1)} = \mathcal{L}(\alpha). 
\]

%%%%%%%%%%%%%%%%%%%%%%%%%%%%%%%%%
\subsection{Four cases analogous to Bessel case}
%%%%%%%%%%%%%%%%%%%%%%%%%%%%%%%%%
In this subsection, we provide a complete characterization of all Laguerre-Hahn forms of class zero analogous to the classical Bessel case, expressed in terms of the recurrence coefficients, the corresponding functional equation, and the coefficients appearing in the Laguerre-Hahn main structure relation. These forms correspond to the canonical setting $\Phi(x)=x^2$, in which case the system \eqref{r0}--\eqref{Fn-System} yields the following result:

\vspace{-0.2cm}
\begin{lemma}\cite{Bouakkaz-these}
If \(\Phi(x) = x^2\), then the coefficients $\beta_n$ and $\gamma_n$ are given by
\begin{align}
\beta_{n+1} &= \frac{\mu(r_0 + 1)}{(2n + 1 - r_0)(2n -1 - r_0)}, \quad r_0 \neq 2n-1, \quad n \geq 0, \label{beta-n-P92}\\
\gamma_{n+1} &= \frac{ \mu^2 (r_0 + 1)^2 - 4(2n - r_0 - 1)^2 \mathbf{R}^2 }{4(2n -2 - r_0)(2n -1 - r_0)^2(2n - r_0)} , \quad r_0 \neq n-1, \quad n \geq 1, \label{gamma-n-P92}
\end{align}
where
\begin{equation}\label{3-83}
\mu := (3 + r_0 - 2b_2)\beta_0 - b_1=(2+a_1)\beta_0-b_1,
\end{equation}
\begin{equation}\label{3-84}
\mathbf{R}^2:=\frac{1}{4}\left(\left(2+a_1\right) \beta_0-b_1\right)^2-\gamma_1\left(a_1+2 b_2-1\right)\left(a_1+b_2+1\right).
\end{equation}
\end{lemma}

In \cite{Bouakkaz-these,Bouakkaz-Maroni-1991}, the authors determined two Laguerre-Hahn forms of class $s=0$ analogous to the classical Bessel case. However, the resolution of the system \eqref{r0}--\eqref{Fn-System} remains incomplete. In the present work, we complete this resolution, leading to two additional families.

Indeed, for $\mathbf{R}=0$, the resolution is complete and the corresponding family is fully described in \cite{Bouakkaz-these}; this corresponds to Case~4 analogous to the Bessel case (see below). In contrast, for $\mathbf{R}\neq 0$, the resolution is incomplete. In this canonical situation, we have
\[
\Phi(x) = x^2, \quad \psi(x) = a_1 x + a_0, \quad B(x) = b_2 x^2 + b_1 x + b_0.
\]
The change of variable yields
\[
\tilde{\phi}(x)=x^2,\quad \tilde{\psi}(x)=a_1 x+a^{-1} \psi(b),\quad \text{and} \quad \tilde{B}(x)=b_2 x^2+a^{-1} B^{\prime}(b)+a^{-2} B(b).
\]
We choose $a$ such that $\mathbf{R}=1$. Then, \eqref{3-84} becomes
\begin{equation}
\frac{\mu^2}{4}-\gamma_1\left(a_1+2 b_2-1\right)\left(a_1+b_2+1\right)=1. 
\end{equation}
\vspace*{-\baselineskip}
\vspace{-0.2cm}
We further set
\begin{align}
\frac{1+r_0}{2}\left(\frac{\mu}{2}-1\right)&=\tau+1,  \label{3-110}\\
-\frac{1+r_0}{2}\left(\frac{\mu}{2}+1\right)&=\tau+2 \alpha-1. \label{3-111}
\end{align}
Thus, three cases arise:
\begin{enumerate}
\item[(i)] $a_1+b_2+1\neq 0$ \; (i.e., $\mu\neq -2$ and $\mu\neq 2$);
\item[(ii)] $a_1+b_2+1=0$ and $\mu=-2$ \; (i.e. $a_1 + b_2 + 1 = 0$ et $\tau = 1 - 2\alpha$);
\item[(iii)] $a_1+b_2+1=0$ and $\mu=2$ \; (i.e. $a_1+b_2+1=0$ et $\tau=-1$).
\end{enumerate}

Only Case~(i) was treated in \cite{Bouakkaz-these}, leading to Case~1 analogous to the Bessel case presented below in Table~\ref{tab1-Bessel}. 

We now complete the resolution of the system by analyzing Cases~(ii) and~(iii).
\begin{enumerate}
\item[$\bullet$] \textbf{Case (ii):} In this case, $a_1=-b_2-1$, $\mu=-2$, and $\tau=1-2\alpha$. Hence, from \eqref{3-110}, we obtain $r_0=2\alpha-3$. Substituting this value into \eqref{beta-n-P92}, we readily obtain the explicit formulas for $\beta_{n+1}$ and $\gamma_{n+1}$, presented in Table~\ref{tab2-Bessel} below. Using \eqref{r0-explicite} together with $a_1=-b_2-1$, we deduce that $b_2=r_0+2=2\alpha-1$, and therefore $a_1=-2\alpha$. From \eqref{3-83}, we obtain
\[
b_1=(2+a_1)\beta_0-\mu=(2-2\alpha)\beta_0+2.
\]

Finally, using the relations
\begin{equation}\label{FE-n=0}
\psi(\beta_0) + B'(\beta_0) = 0,
\end{equation}
and
\begin{equation}\label{FE-n=1}
\Phi(\beta_0) - B(\beta_0)
= \gamma_1\bigl(a_1 + 2b_2 - 1\bigr),
\end{equation}
which are obtained from the functional equation for $n=0$ and $n=1$, we get
\[
a_0=-(a_1+2b_2)\beta_0-b_1=-2,
\]
and
\[
b_0=(1-b_2)\beta_0^2-b_1\beta_0-\gamma_1\bigl(a_1+2b_2-1\bigr)
=-2\beta_0-\gamma_1(2\alpha-3).
\]
This completes the analysis of Case~(ii).

\item[$\bullet$] \textbf{Case (iii):} Assume that $a_1=-b_2-1$, $\mu=2$, and $\tau=-1$. Then, from \eqref{3-111}, it follows that $r_0=-2\alpha+1$. Substituting this value into \eqref{beta-n-P92}, we directly obtain the explicit expressions of $\beta_{n+1}$ and $\gamma_{n+1}$, which are reported in Table~\ref{tab3-Bessel} below. Combining \eqref{r0-explicite} with the relation $a_1=-b_2-1$, we infer that $b_2=r_0+2=-2\alpha+3$, and hence $a_1=2\alpha-4$. Moreover, from \eqref{3-83}, we deduce
\[
b_1=(2+a_1)\beta_0-\mu=(2\alpha-2)\beta_0-2.
\]

Finally, applying \eqref{FE-n=0} and \eqref{FE-n=1}, we obtain
\[
a_0=-(a_1+2b_2)\beta_0-b_1=2,
\]
and
\[
b_0=(1-b_2)\beta_0^2-b_1\beta_0-\gamma_1\bigl(a_1+2b_2-1\bigr)
=2\beta_0-\gamma_1(1-2\alpha).
\]
This concludes the study of Case~(iii).
\end{enumerate}

As a consequence of this complete study, we obtain precisely four families analogous to the Bessel case:
\begin{enumerate}
\item[$\bullet$] Case 1: $\mathbf{R}\neq 0,\quad a_1+b_2+1\neq 0$.
\item[$\bullet$] Case 2 ({\bf new}): $\mathbf{R} \neq 0, \quad a_1 + b_2 + 1 = 0,\quad \mu=-2$.
\item[$\bullet$] Case 3 ({\bf new}): $
\mathbf{R}\neq 0, \quad a_1+b_2+1=0, \quad \mu=2$.
\item[$\bullet$] Case 4: $\mathbf{R}=0$.
\end{enumerate}

%%%%%%%%%%%%%%%%%%%%%%%%%%%%%%%%%
\subsubsection{Case 1 analogous to Bessel}
%%%%%%%%%%%%%%%%%%%%%%%%%%%%%%%%%
\begin{table}[htbp]
\centering
\renewcommand{\arraystretch}{1.5}
\setlength{\tabcolsep}{1em} 

\begin{tabular}{@{}p{0.22\textwidth} p{0.74\textwidth}@{}}

\midrule
\begin{minipage}[t]{0.22\textwidth}
\raggedright \textbf{Regularity conditions}
\end{minipage}
&
\begin{minipage}[t]{0.74\textwidth}
\raggedright
$\lambda, \rho, \tau, \alpha \in \mathbb{C}, \quad
\rho \neq 0, \quad \tau+1 \neq -n$, \\ $\tau+2\alpha-1 \neq -n, \quad \tau+\alpha \neq -n/2, \quad n\geq 0$.
\end{minipage}
\\
\midrule

\begin{minipage}[t]{0.22\textwidth}
\raggedright \textbf{The recurrence coefficients}
\end{minipage}
&
\begin{minipage}[t]{0.8\textwidth}
\raggedright
$\beta_0 = \displaystyle\frac{1-\alpha}{(\tau+\alpha)(\tau+\alpha-1)} + \lambda,~~
\beta_{n+1} = \displaystyle\frac{1-\alpha}{(n+\tau+\alpha)(n+\tau+\alpha+1)}, \quad n\geq 0$, \\
$\gamma_1 = -\rho \displaystyle\frac{(\tau+1)(\tau+2\alpha-1)}{(2\tau+2\alpha-1)(\tau+\alpha)^2(2\tau+2\alpha+1)}$,\\
$\gamma_{n+1} = \displaystyle\frac{(n+\tau+1)(n+\tau+2\alpha-1)}{(2n+2\tau+2\alpha-1)(n+\tau+\alpha)^2(2n+2\tau+2\alpha+1)}, \quad n\geq 1$.
\end{minipage}
\\
\midrule

\begin{minipage}[t]{0.22\textwidth}
\raggedright \textbf{The functional equation}
\end{minipage}
&
\begin{minipage}[t]{0.74\textwidth}
\raggedright
$\Phi(x)=x^2$, \\
$\psi(x)=2\displaystyle\left\{\frac{1-\rho}{\rho}+\frac{\rho-2}{\rho}(\tau+\alpha)\right\}x+\frac{2}{\rho}(2(\tau+\alpha)-1)\beta_0-2\frac{1-\alpha}{\tau+\alpha}$, \\
$B(x)=\displaystyle\frac{1-\rho}{\rho}\left\{2(\tau+\alpha)-1\right\}x^2+2\left\{\frac{1-\alpha}{\tau+\alpha}+\beta_0\left[(\tau+\alpha)\frac{\rho-2}{\rho}+\frac{1}{\rho}\right]\right\}x$\\
$\quad\quad\quad~~+\displaystyle\left(\frac{2\tau+2\alpha-1}{\rho}\right)\beta_0^2 + 2\frac{\alpha-1}{\tau+\alpha}\beta_0 - \frac{\rho(\tau+1)(\tau+2\alpha-1)}{[2(\tau+\alpha)-1](\tau+\alpha)^2}$.
\end{minipage}
\\
\midrule

\begin{minipage}[t]{0.22\textwidth}
\raggedright \textbf{The coefficients of the Laguerre-Hahn structure relation}
\end{minipage}
&
\begin{minipage}[t]{0.74\textwidth}
\raggedright
$C_0(x) = 2\displaystyle\left[(\tau+\alpha)\left(\frac{2}{\rho}-1\right) - \frac{1}{\rho}\right]x - \frac{2}{\rho}\big(2(\tau+\alpha)-1\big)\beta_0 - 2\frac{\alpha-1}{\tau+\alpha}$,\\
$C_{n+1}(x) = \displaystyle\big[2n + 2(\tau+\alpha)\big]x - \frac{2(1-\alpha)}{n+\tau+\alpha}, \quad n\geq 0$, \\
$D_0(x) = \displaystyle\frac{1}{\rho}(2\tau+2\alpha-1),\quad D_{n+1}(x)= 2(n+\tau+\alpha)+1, \quad n\geq 0$.
\end{minipage}
\\
\bottomrule
\end{tabular}
\caption{Case 1 analogous to Bessel}
\label{tab1-Bessel}
\end{table}
\FloatBarrier
\vspace*{-\baselineskip}
\vspace{-0.2cm}
When \(\tau \in \mathbb{N}\), the form \(u_0\) can be read as 
\[
u_0=\mathcal{B}^{(\tau)}(\alpha)
\left(\lambda;\,{0\atop\rho}\,;1\right), 
\]
where $\mathcal{B}(\alpha)$ represents the classical Bessel case.

%%%%%%%%%%%%%%%%%%%%%%%%%%%%%%%%%
\subsubsection{Case 2 analogous to Bessel}
%%%%%%%%%%%%%%%%%%%%%%%%%%%%%%%%%
\begin{table}[h!]
\centering
\renewcommand{\arraystretch}{1.5}
\setlength{\tabcolsep}{1em} 
\begin{tabular}{@{}p{0.22\textwidth} p{0.7\textwidth}@{}}

\midrule
\begin{minipage}[t]{0.22\textwidth}
\raggedright \textbf{Regularity conditions}
\end{minipage}
&
\begin{minipage}[t]{0.7\textwidth}
\raggedright
$\lambda,\rho,\alpha \in \mathbb{C},  \quad \rho \neq 0, \quad \alpha \neq \displaystyle\frac{n+3}{2}, \quad n\geq 0$.
\end{minipage}
\\
\midrule

\begin{minipage}[t]{0.22\textwidth}
\raggedright \textbf{The recurrence coefficients}
\end{minipage}
&
\begin{minipage}[t]{0.7\textwidth}
\raggedright
$\beta_0 = \lambda, \quad
\beta_{n+1} = \displaystyle\frac{1-\alpha}{(n+1-\alpha)(n+2-\alpha)}, \quad n\geq 0$, \\
$\gamma_1 = \rho, \quad
\gamma_{n+1} = -\displaystyle\frac{n(n+2-2\alpha)}{(2n+1-2\alpha)(n+1-\alpha)^2(2n+3-2\alpha)}, ~~ n\geq 1$. 
\end{minipage}
\\
\midrule

\begin{minipage}[t]{0.22\textwidth}
\raggedright \textbf{The functional equation}
\end{minipage}
&
\begin{minipage}[t]{0.7\textwidth}
\raggedright
$\Phi(x) = x^2, \quad \psi(x) = -2\alpha x - 2$, \\
$B(x) = (2\alpha-1)x^2 + \{2(1-\alpha)\lambda + 2\}x - 2\lambda - \rho(2\alpha-3)$.
\end{minipage}
\\
\midrule

\begin{minipage}[t]{0.22\textwidth}
\raggedright \textbf{The coefficients of the Laguerre-Hahn structure relation}
\end{minipage}
&
\begin{minipage}[t]{0.7\textwidth}
\raggedright
$C_0(x) = 2(\alpha-1)x + 2, \quad
C_{n+1}(x) = 2(n - \alpha + 1)x - \displaystyle\frac{2(1-\alpha)}{n-\alpha+1}, \quad n\geq 0$, \\
$D_0(x) = 0, \quad D_{n+1}(x) = 2n-2\alpha + 3, \quad n\geq 0$.
\end{minipage}
\\
\bottomrule
\end{tabular}
\caption{Case 2 analogous to Bessel (new case)}
\label{tab2-Bessel}
\end{table}
\FloatBarrier
\vspace*{-\baselineskip}
\vspace{-0.2cm}
The form \(u_0\) can be read as 
\[
u_0^{(1)} = h_{-1}\mathcal{B}(2-\alpha). 
\]

%%%%%%%%%%%%%%%%%%%%%%%%%%%%%%%%%
\subsubsection{Case 3 analogous to Bessel}
%%%%%%%%%%%%%%%%%%%%%%%%%%%%%%%%%
\vspace*{-\baselineskip}
\vspace{-0.2cm}
\begin{table}[!htbp]
\centering
\renewcommand{\arraystretch}{1.5}
\setlength{\tabcolsep}{1em}

\begin{tabular}{@{}p{0.22\textwidth} p{0.7\textwidth}@{}}

\midrule
\begin{minipage}[t]{0.22\textwidth}
\raggedright \textbf{Regularity conditions}
\end{minipage}
&
\begin{minipage}[t]{0.7\textwidth}
\raggedright
$\lambda, \rho, \alpha \in\mathbb{C}, \quad \rho\neq 0, \quad \alpha\neq \displaystyle\frac{1-n}{2}, \quad n\geq 0$.
\end{minipage}
\\
\midrule

\begin{minipage}[t]{0.22\textwidth}
\raggedright \textbf{The recurrence coefficients}
\end{minipage}
&
\begin{minipage}[t]{0.7\textwidth}
\raggedright
$\beta_0=\lambda, \quad
\beta_{n+1}=\displaystyle\frac{1-\alpha}{(n+\alpha-1)(n+\alpha)}, \quad n\geq 0$, \\
$\gamma_1=\rho, \quad \gamma_{n+1}=\displaystyle\frac{n(n+2\alpha-2)}{(2n+2\alpha-3)(n+\alpha-1)^2(2n+2\alpha-1)}, \quad n\geq 1$. 
\end{minipage}
\\
\midrule

\begin{minipage}[t]{0.22\textwidth}
\raggedright \textbf{The functional equation}
\end{minipage}
&
\begin{minipage}[t]{0.7\textwidth}
\raggedright
$\Phi(x)=x^2, \quad
\psi(x)=2(\alpha-2)x+2$, \\
$B(x)=-(2\alpha-3)x^2+2\{(\alpha-1)\lambda-1\}x+2\lambda+\rho(2\alpha-1)$.
\end{minipage}
\\
\midrule

\begin{minipage}[t]{0.22\textwidth}
\raggedright \textbf{The coefficients of the Laguerre-Hahn structure relation}
\end{minipage}
&
\begin{minipage}[t]{0.7\textwidth}
\raggedright
$C_0(x)=2(1-\alpha)x-2$, \\
$C_{n+1}(x)=2(n+\alpha-1)x-2\displaystyle\frac{1-\alpha}{n+\alpha-1},  \quad n\geq 0$,\\
$D_0(x)=0. \quad D_{n+1}(x)=2n+2\alpha-1, \quad n\geq 0$.
\end{minipage}
\\
\bottomrule
\end{tabular}
\caption{Case 3 analogous to Bessel (new case)}
\label{tab3-Bessel}
\end{table}
\FloatBarrier
\vspace*{-\baselineskip}
\vspace{-0.4cm}
The form \(u_0\) can be read as 
\[
u_0^{(1)}=\mathcal{B}(\alpha). 
\]

%%%%%%%%%%%%%%%%%%%%%%%%%%%%%%%%%
\subsubsection{Case 4 analogous to Bessel}
%%%%%%%%%%%%%%%%%%%%%%%%%%%%%%%%%
\vspace*{-\baselineskip}
\vspace{-0.2cm}
\begin{table}[htbp]
\centering
\renewcommand{\arraystretch}{1.5}
\setlength{\tabcolsep}{1em} 

\begin{tabular}{@{}p{0.22\textwidth} p{0.64\textwidth}@{}}

\midrule
\begin{minipage}[t]{0.22\textwidth}
\raggedright \textbf{Regularity conditions}
\end{minipage}
&
\begin{minipage}[t]{0.64\textwidth}
\raggedright
$\lambda, \rho, \alpha \in\mathbb{C}, \quad \rho\neq0, \quad \alpha\neq \displaystyle\frac{1-n}{2}, \quad n\geq-1$.
\end{minipage}
\\
\midrule

\begin{minipage}[t]{0.22\textwidth}
\raggedright \textbf{The recurrence coefficients}
\end{minipage}
&
\begin{minipage}[t]{0.64\textwidth}
\raggedright
$\beta_0=\displaystyle\frac{1}{(\alpha-1)\alpha}+\lambda, \quad
\beta_{n+1}=\displaystyle\frac{1}{(n+\alpha)(n+\alpha+1)}, \quad n\geq0$,\\
$\gamma_1=\displaystyle\frac{\rho}{(2\alpha-1)\alpha^2(2\alpha+1)}$, \\ 
$\gamma_{n+1}=\displaystyle\frac{1}{(2n+2\alpha-1)(n+\alpha)^2(2n+2\alpha+1)}, \quad n\geq1$.
\end{minipage}
\\
\midrule

\begin{minipage}[t]{0.22\textwidth}
\raggedright \textbf{The functional equation}
\end{minipage}
&
\begin{minipage}[t]{0.64\textwidth}
\raggedright
$\Phi(x)=x^2,\quad
\psi(x)=2\displaystyle\left\{\frac{1-\rho}{\rho}+\alpha\frac{\rho-2}{\rho}\right\} x+\frac{2}{\rho}(2\alpha-1)\beta_0-\frac{2}{\alpha}$,\\
$B(x)=\displaystyle\frac{\rho-1}{\rho}(1-2\alpha)x^2+2\left\{\left[\alpha+\frac{1}{\rho}(1-2\alpha)\right]\beta_0+\frac{1}{\alpha}\right\} x$\\
$\quad\quad\quad~~
+ \displaystyle\frac{2\alpha-1}{\rho}\beta_0^2 - \frac{2\beta_0}{\alpha} + \frac{\rho}{(2\alpha-1)\alpha^2}$.
\end{minipage}
\\
\midrule

\begin{minipage}[t]{0.22\textwidth}
\raggedright \textbf{The coefficients of the Laguerre-Hahn structure relation}
\end{minipage}
&
\begin{minipage}[t]{0.64\textwidth}
\raggedright
$C_0(x)=2\displaystyle\left(\frac{2-\rho}{\rho}\alpha-\frac{1}{\rho}\right)x-\frac{2}{\rho}(2\alpha-1)\beta_0+\frac{2}{\alpha}$, \\
$C_{n+1}(x)=2(n+\alpha)x-\displaystyle\frac{2}{n+\alpha}, \quad n\geq 0$, \\
$D_0(x)=\displaystyle\frac{1}{\rho}(2\alpha-1), \quad D_{n+1}(x)=2n+2\alpha+1, \quad n\geq 0$.
\end{minipage}
\\
\bottomrule
\end{tabular}
\caption{Case 4 analogous to Bessel}
\label{tab4-Bessel}
\end{table}
\FloatBarrier
\vspace*{-\baselineskip}
\vspace{-1cm}
\begin{remark}
The family described in Table~\ref{tab4-Bessel} coincides, up to a shift, with the family obtained in \cite{Bouakkaz-these} for the case $\mathbf{R}=0$, which is characterized by the recurrence coefficients (3-93)--(3-96) (see page~89 of \cite{Bouakkaz-these}). More precisely, the form $u_0$ in Table~\ref{tab4-Bessel} and the form denoted by $\mathcal{B}_2(\alpha,\lambda,\rho)$ in \cite{Mohamed-Imed-2025,Bouakkaz-these}, which is defined through the recurrence coefficients (3-93)--(3-96), are related by
\[
u_0 = h_{2}\left(\mathcal{B}_2\left(2\alpha,\frac{\lambda}{2},\rho\right)\right).
\]
\end{remark}

\vspace{-0.6cm}
%%%%%%%%%%%%%%%%%%%%%%%%%%%%%%%%%
\subsection{Two cases analogous to Jacobi case}
%%%%%%%%%%%%%%%%%%%%%%%%%%%%%%%%%
In this subsection, we provide a full and complete description of all Laguerre-Hahn forms of class zero analogous to the classical Jacobi case, corresponding to the canonical setting $\Phi(x)=x^2-1$.
\begin{lemma}\cite{Bouakkaz-these}
If \(\Phi(x) = x^2-1\), then the coefficients $\beta_n$ and $\gamma_n$ are given by
\begin{align}
%\beta_1 &= \frac{\mu}{r_0 - 1}, \\
\beta_{n+1} &= \frac{\mu(r_0 + 1)}{(2n - 1 - r_0)(2n + 1 - r_0)}, \quad r_0 \neq 2n-1, \, n \geq 1, \\
\gamma_{n+1} &= \frac{n(n - 1 - r_0)\Big((2n -1 - r_0)^2 - \mu^2\Big) + (2n - 1 - r_0)^2\, \mathbf{T}}{(2n - r_0)(2n - 1 - r_0)^2(2n -2 - r_0)} , \quad r_0\neq n-1, \, n \geq 1,
\end{align}
where
$$
\mu := (3 + r_0 - 2b_2)\beta_0 - b_1=(2+a_1)\beta_0-b_1,
$$
$$
\mathbf{T}:=\gamma_1(a_1+2b_2-1)(a_1+b_2+1).
$$
\end{lemma}
In \cite{Bouakkaz-these,Bouakkaz-Maroni-1991}, the authors determined all strict Laguerre-Hahn forms of class $s=0$ analogous to the classical Jacobi case. These are precisely two families:
\begin{enumerate}
\item[$\bullet$] Case 1: $\mathbf{T}\neq 0$.
\item[$\bullet$] Case 2: $\mathbf{T} = 0$.
\end{enumerate}

%%%%%%%%%%%%%%%%%%%%%%%%%%%%%%%%%
\subsubsection{Case 1 analogous to Jacobi}
%%%%%%%%%%%%%%%%%%%%%%%%%%%%%%%%%
\vspace*{-\baselineskip}
\vspace{-0.2cm}
\begin{table}[htbp]
\centering
\renewcommand{\arraystretch}{1.5}
\setlength{\tabcolsep}{1em} 

\begin{tabular}{@{}p{0.22\textwidth} p{0.8\textwidth}@{}}

\midrule
\begin{minipage}[t]{0.22\textwidth}
\raggedright \textbf{Regularity conditions}
\end{minipage}
&
\begin{minipage}[t]{0.8\textwidth}
\raggedright
$\lambda, \rho, \alpha, \beta, \tau \in\mathbb{C},\quad \rho\neq0$,\\
$\tau\neq-n-1,\quad \tau+\alpha\neq-n-1,\quad \tau+\beta\neq-n-1,\quad 2\tau+\alpha+\beta\neq-n,\quad n\geq 0$.
\end{minipage}
\\
\midrule

\begin{minipage}[t]{0.22\textwidth}
\raggedright \textbf{The recurrence coefficients}
\end{minipage}
&
\begin{minipage}[t]{0.8\textwidth}
\raggedright
$\beta_0=\displaystyle\frac{\alpha^2-\beta^2}{(\alpha+\beta+2\tau+2)(\alpha+\beta+2 \tau)}+\lambda$,\\
$\beta_{n+1}=\displaystyle\frac{\alpha^2-\beta^2}{(2n+\alpha+\beta+2\tau+2)(2n+\alpha +\beta+2\tau+4)},\quad n\geq 0$,\\
$\gamma_1=4\rho\displaystyle\frac{(\tau+1)(\tau+\alpha+\beta+1)(\tau+\alpha+1)(\tau+\beta+1)}{(2 \tau+\alpha+\beta+1)(2\tau+\alpha+\beta+2)^2(2\tau+\alpha+\beta+3)}$,\\
$\gamma_{n+1}=4\displaystyle\frac{(n+\tau+1)(n+\tau+\alpha+\beta+1)(n+\tau+\alpha+1)(n+ \tau+\beta+1)}{(2n+2\tau+\alpha+\beta+1)(2n+2\tau+\alpha+\beta+2)^2(2n+2\tau+\alpha+ \beta+3)},\quad n\geq1$.
\end{minipage}
\\
\midrule

\begin{minipage}[t]{0.22\textwidth}
\raggedright \textbf{The functional equation}
\end{minipage}
&
\begin{minipage}[t]{0.8\textwidth}
\raggedright
$\Phi(x)=x^2-1$,\\
$\psi(x)=\left\{\frac{\rho-2}{\rho}(2\tau+\alpha+\beta+1)-1\right\}x+\frac{2}{\rho}(2\tau+\alpha+\beta+1)\beta_0-\frac{\alpha^2-\beta^2}{2\tau+\alpha+\beta+2}$,\\
$B(x)=\frac{1-\rho}{\rho}(2\tau+\alpha+\beta+1)x^2+\left\{\left[\frac{\rho-2}{\rho}(2\tau+\alpha+\beta)+2\frac{\rho-1}{\rho}\right]\beta_0+\frac{\alpha^2-\beta^2}{2\tau+\alpha+\beta+2}\right\}x$\\
$\quad\quad\quad+(2\tau+\alpha+\beta+3)\gamma_1-1+\frac{1}{\rho}(2\tau+\alpha+\beta+1)\beta_0^2-\frac{\alpha^2-\beta^2}{2\tau+\alpha+\beta+2}\beta_0$.
\end{minipage}
\\
\midrule

\begin{minipage}[t]{0.22\textwidth}
\raggedright \textbf{The coefficients of the Laguerre-Hahn structure relation}
\end{minipage}
&
\begin{minipage}[t]{0.8\textwidth}
\raggedright
$C_0(x)=\left\{\frac{2-\rho}{\rho}(2\tau+\alpha+\beta+1)-1\right\}x-\frac{2}{\rho}(2\tau+\alpha+\beta+1)\beta_0+\frac{\alpha^2-\beta^2}{2\tau+\alpha+\beta+2}$, \\
$C_{n+1}(x)=(2n+2\tau+\alpha+\beta+2)x
-\frac{\alpha^2-\beta^2}{2n+2\tau+\alpha+\beta+2},
\quad n\geq 0$,\\
$D_0(x)=\frac{1}{\rho}(2\tau+\alpha+\beta+1),\quad
D_{n+1}(x)=2n+2\tau+\alpha+\beta+3,\quad n\geq 0$.
\end{minipage}
\\
\bottomrule
\end{tabular}
\caption{Case 1 analogous to Jacobi}
\label{tab1-Jacobi}
\end{table}
\FloatBarrier
\vspace*{-\baselineskip}
\vspace{-0.2cm}
When \(\tau\in\mathbb{N}\) the form \(u_0\)  can be read as 
\[
u_0=\mathcal{J}^{(r)}(\alpha,\beta)\left(\lambda;\,{0\atop\rho}\,;1\right), 
\]
where $\mathcal{J}(\alpha,\beta)$ represents the classical Jacobi form.

%%%%%%%%%%%%%%%%%%%%%%%%%%%%%%%%%
\subsubsection{Case 2 analogous to Jacobi}
%%%%%%%%%%%%%%%%%%%%%%%%%%%%%%%%%

\begin{table}[htbp]
\centering
\renewcommand{\arraystretch}{1.5}
\setlength{\tabcolsep}{1em} 

\begin{tabular}{@{}p{0.22\textwidth} p{0.75\textwidth}@{}}

\midrule
\begin{minipage}[t]{0.22\textwidth}
\raggedright \textbf{Regularity conditions}
\end{minipage}
&
\begin{minipage}[t]{0.75\textwidth}
\raggedright
$\lambda, \rho, \alpha, \beta \in \mathbb{C}, \quad \rho \neq 0, \quad \alpha \neq -n, \quad \beta \neq -n, \quad \alpha + \beta \neq -n, \quad n \geq 1$.
\end{minipage}
\\
\midrule

\begin{minipage}[t]{0.22\textwidth}
\raggedright \textbf{The recurrence coefficients}
\end{minipage}
&
\begin{minipage}[t]{0.75\textwidth}
\raggedright
$\beta_0 = \lambda, \quad
\beta_{n+1} = -\displaystyle\frac{\alpha^2 - \beta^2}{(2n + \alpha + \beta)(2n + \alpha + \beta + 2)}, \quad n \geq 0$, \\
$\gamma_1 = \rho, ~~
\gamma_{n+1} = 4\displaystyle\frac{n(n + \alpha + \beta)(n + \alpha)(n + \beta)}{(2n + \alpha + \beta - 1)(2n + \alpha + \beta)^2(2n + \alpha + \beta + 1)}, ~~ n \geq 1$.
\end{minipage}
\\
\midrule

\begin{minipage}[t]{0.22\textwidth}
\raggedright \textbf{The functional equation}
\end{minipage}
&
\begin{minipage}[t]{0.75\textwidth}
\raggedright
$\Phi(x) = x^2 - 1, \quad
\psi(x) = (\alpha + \beta - 2)x + \alpha - \beta$, \\
$B(x) = (1 - \alpha - \beta)x^2 + ((\alpha + \beta)\lambda + \beta - \alpha)x + (\alpha - \beta)\lambda - 1 + \rho(1 + \alpha + \beta)$.
\end{minipage}
\\
\midrule

\begin{minipage}[t]{0.22\textwidth}
\raggedright \textbf{The coefficients of the Laguerre-Hahn structure relation}
\end{minipage}
&
\begin{minipage}[t]{0.75\textwidth}
\raggedright
$C_0(x) = -(\alpha + \beta)x + \beta - \alpha$,\\ 
$C_{n+1}(x) = (2n + \alpha + \beta)x + \displaystyle\frac{\alpha^2 - \beta^2}{2n + \alpha + \beta}, \quad n \geq 0$, \\
$D_0(x) = 0, \quad D_{n+1}(x) = 2n + \alpha + \beta +1, \quad n \geq 0$.
\end{minipage}
\\
\bottomrule
\end{tabular}
\caption{Case 2 analogous to Jacobi}
\label{tab2-Jacobi}
\end{table}
\FloatBarrier
\vspace*{-\baselineskip}
\vspace{-0.2cm}
The form \(u_0\) can be read as 
\[
u_0^{(1)} = h_{-1}\mathcal{J}(\alpha,\beta). 
\]
\begin{remark}
The family described in Table~\ref{tab2-Jacobi} coincides with the family obtained in \cite{Bouakkaz-these} for the case $\mathbf{R}=0$, which is characterized by the recurrence coefficients (3-126)--(3-129) (see page~102 of \cite{Bouakkaz-these}). More precisely, the form $u_0$ in Table~\ref{tab2-Jacobi} and the form denoted by $\mathcal{J}_2(\alpha, \mu, \lambda, \rho)$ in \cite{Mohamed-Imed-2025,Bouakkaz-these}, which is defined via the recurrence coefficients (3-126)--(3-129), are related by
\[
u_0 = \mathcal{J}_2(\alpha+\beta, \alpha-\beta, \lambda, \rho).
\]
\end{remark}

%%%%%%%%%%%%%%%%%%%%%%%%%%%%%%%%%%%%%%%%%%
\section{Final remarks}

The present work reveals that the Laguerre-Hahn set of class zero is richer than previously documented, as two additional families analogous to the Bessel case emerge from a complete resolution of the governing system. This finding, together with the corrections brought to the structure relation coefficients, consolidates the theoretical foundations on which the algorithmic approach \textit{4oDELH} relies. The effectiveness of this symbolic framework, illustrated here for the second Bessel-type family, opens a systematic path toward the explicit construction of fourth-order differential equations for all remaining Laguerre-Hahn families, in particular, those of class zero, whose treatment is already established and will be addressed in forthcoming works.

%%%%%%%%%%%%%%%%%%%%%%%%%%%%%%%%%%%%%%%%%%

\section*{Addendum}

In this paper, we present the results of a scientific work originally
started by the two last authors, Pascal Maroni and Zélia da Rocha.
This project coupled theoretical results
with their implementation in a {\it Mathematica$^{\circledR}$} software.
During the preparation, and thanks to the software developed, these authors
realized that there were inconsistencies in some formulas.
Pascal Maroni then undertook a verification of the results presented in \cite{Bouakkaz-these,Bouakkaz-Maroni-1991} and realized that the description of the class 0 of Laguerre–Hahn families given therein was incomplete: two new families, analogous to Bessel, emerged. Moreover, he found it more appropriate to modify the parametrizations of two other sequences. In addition, the coefficients appearing in the main structure relations for all eight families reported in \cite{Bouakkaz-these} were incorrect. Since these coefficients were crucial for deriving the corresponding differential equations, he decided to correct them also.

As is often done by scientists, the work was suspended with the intention to resume it later with a rested mind. 
When the two authors decided to return to this project, Pascal Maroni realized that his handwritten manuscripts had been lost. He only retained a draft containing some of the final formulas in a TeX file. As a consequence, it became necessary to reconstruct the entire work from the beginning.

Pascal Maroni was a very prolific researcher and had other projects and collaborations
prevented him from returning to this subject.
Unfortunately, Pascal Maroni passed away in January 2024. 
His friends, collaborators, and admirers thus decided to honor his memory,
to finish all the ongoing works and publish them to disseminate his ideas into
the scientific community, and to insert Pascal as one of the authors.
In this spirit, Zélia da Rocha invited the first author, Mohamed Khalfallah, to reconstruct the theoretical part so that this work could finally be completed.

\section*{Declarations}

\noindent  {\bf Data Availability} No datasets were generated or analysed during the current study.

\noindent  {\bf Conflicts of Interest} The authors have no conflicts of interest to declare.

\noindent  {\bf Competing interests} The authors declare no competing interests.

\noindent {\bf Funding:} The third author was partially supported by CMUP, a member of LASI, which is financed by national funds through FCT -- Funda\c c\~ao para a Ci\^encia e a Tecnologia, I.P., under the projects with reference UID/00144/2025 and associated DOI given by \url{https://doi.org/10.54499/UID/00144/2025}
%%%%%%%%%%%%%%%%%%%%%%%%%%%%%%%%%%%%%%%
%%%%%%%%%%%%%%%%%%%%%%%%%%%%%%%%%%%%%%%
\begin{appendices}
%%%%%%%%%%%%%%%%%%%%%%%%%%%%%%%%%%%%%%%
%%%%%%%%%%%%%%%%%%%%%%%%%%%%%%%%%%%%%%%

%%%%%%%%%%%%%%%%%%%%%%%%%%%%%%%%%
\section{Four structure relations and a fourth-order differential equation for case 2 analogous to Bessel}\label{Section5}
%%%%%%%%%%%%%%%%%%%%%%%%%%%%%%%%%

In the sequel of the recent paper \cite{khalfallah2026}, we present in this appendix four structure relations and a fourth-order homogeneous linear differential equation satisfied by the Laguerre-Hahn orthogonal polynomial families analogous to the Bessel case, namely Case~2. These results are obtained by applying the algorithm {\it 4oDELH} introduced in \cite{khalfallah2026}.
The input data for the algorithm consist of the recurrence coefficients, the coefficients of the functional equation, and the coefficients of the structure relation given in Table \ref{tab2-Bessel}. From the results of this case, we also illustrate the type of explicit relations and equations satisfied by Laguerre-Hahn families of class 0. These formulas were derived through extensive computations that would be impractical to carry out by hand.

%%%%%%%%%%%%%%%%%%%%%%%%%%%%%%%%%%%
%%%%%%%%%%%%%%%%%%%%%%%%%%%%%%%%%%%
\subsection{Four structure relations}
%%%%%%%%%%%%%%%%%%%%%%%%%%%%%%%%%%%
%%%%%%%%%%%%%%%%%%%%%%%%%%%%%%%%%%%

%%%%%%%%%%%%%%%%%%%%%%%%%%%%%%%%%%%%%%%%%%%%%%%
\noindent {\bf First structure relation}
%%%%%%%%%%%%%%%%%%%%%%%%%%%%%%%%%%%%%%%%%%%%%%%%%%
\begin{eqnarray}
{\bf {\bf G_{0,1}(x;n)}P^{(1)}_{n-1}(x)+{\bf G_{1,1}(x;n)}P^{(1)}_{n}(x)+{\bf H_{1}(x;n)}P_{n}(x)=}&&\notag\\
{\bf {\bf \Phi(x)} P'_{n+1}(x)+{\bf M_{0,1}(x;n)}P_{n+1}(x).}&&\notag
\end{eqnarray}
\begin{eqnarray}
&& {\bf G_{0,1}(x;n)}=0,\ 
 {\bf G_{1,1}(x;n)} =(2 \alpha -1) {\bf x^2}+ (2 (1-\alpha ) \lambda +2){\bf x}-(2 \alpha -3) \rho -2 \lambda ;\notag\\
&& 
{\bf H_{1}(x;0)}=(2 \alpha -3) \rho,\  {\bf H_{1}(x;n)}=\frac{n (-2 \alpha +n+2)}{(-2 \alpha +2 n+1) (-\alpha +n+1)^2},\ n\geq 1;\notag\\
&&{\bf \Phi(x)}={\bf x^2},\  {\bf M_{0,1}(x;n)}={\bf x} (2 \alpha -n-2)+\frac{2 \alpha -n-2}{\alpha -n-1}. \notag
\end{eqnarray}

%%%%%%%%%%%%%%%%%%%%%%%%%%%%%%%%%%%%%%%%%%%%%%%
\noindent {\bf Second structure relation}
%%%%%%%%%%%%%%%%%%%%%%%%%%%%%%%%%%%%%%%%%%%%%%%
\begin{eqnarray}
{\bf {\bf G_{0,2}(x;n)}P^{(1)}_{n-1}(x)+ {\bf G_{1,2}(x;n)}P^{(1)}_{n}(x)+ {\bf H_{2}(x;n)}P_{n}(x)=}&&\notag\\
{\bf {\bf \Phi^2(x)}P'_{n+1}(x)+{\bf M_{1,2}(x;n)}P'_{n+1}(x)+{\bf M_{0,2}(x;n)}P_{n+1}(x).}&&\notag
\end{eqnarray}
\begin{eqnarray}
&& {\bf G_{0,2}(x;0)}= 2 (2 \alpha -3) \rho\Big(  (2 \alpha -1) {\bf x^2}+ (2-2 (\alpha -1) \lambda ){\bf x}+(3-2 \alpha )\rho-2 \lambda  \Big), \notag\\
&& {\bf G_{0,2}(x;n)} = \frac{2 n (-2 \alpha +n+2)}{(-2 \alpha +2 n+1) (-\alpha +n+1)^2}\notag\\
&&\hspace{2.05cm} \Big(
(2 \alpha -1) {\bf x^2}+ 2(1- (\alpha -1) \lambda ){\bf x}+(3-2 \alpha ) \rho -2 \lambda 
 \Big),\ n\geq 1;\notag\\
&& {\bf G_{1,2}(x;n)} = \frac{1}{(-\alpha +n+1)}\Big({\bf x^3}(2 \alpha -1)\left(-2 (\alpha -1) +n^2-(\alpha -3) n\right) +\notag\\
&&{\bf x^2}\left(2 (\alpha -1)^2 \lambda -2 (\alpha -1)+n^2 (2-2 (\alpha -1)
   \lambda )+n (2 (\alpha -2) (\alpha -1) \lambda +3)\right)+\notag\\
&&{\bf x}\left(n^2 ((3-2 \alpha )\rho -2 \lambda )+n ((\alpha -1) (2 \alpha -3) \rho +2)\right)+
n ((3-2 \alpha ) \rho -2 \lambda )\Big),\notag\\
&& {\bf H_{2}(x;n)}= -\frac{n^2 (-2 \alpha +n+2) ({\bf x} (n-\alpha  +1)+1)}{(-2 \alpha +2 n+1) (-\alpha +n+1)^3},\quad 
{\bf \Phi^2(x)}={\bf x^4}, \notag\\
&& {\bf M_{1,2}(x;n)} = {\bf x^3 }(2 \alpha -n)+\frac{{\bf x^2 }(-2 \alpha +n+2)}{(-\alpha +n+1)},\notag\\
&& {\bf M_{0,2}(x;n)} = -\frac{-2 \alpha +n+2}{(-\alpha +n+1)^2} \Big({\bf x^2} \left((\alpha -1)^2+n^2-2 (\alpha -1) n\right)+n \Big).\notag 
\end{eqnarray}

%%%%%%%%%%%%%%%%%%%%%%%%%%%%%%%%%%%%%%%%%%%%%%%
\noindent {\bf Third structure relation}
%%%%%%%%%%%%%%%%%%%%%%%%%%%%%%%%%%%%%%%%%%%%%%%
\begin{eqnarray}
{\bf {\bf G_{0,3}(x;n)}P^{(1)}_{n-1}(x)+{\bf G_{1,3}(x;n)}P^{(1)}_{n}(x)+{\bf H_{3}(x;n)}P_{n}(x)=}&&\notag\\
{\bf{\bf \Phi^3(x)}P^{(3)}_{n+1}(x)+{\bf M_{2,3}(x;n)}P''_{n+1}(x)+{\bf M_{1,3}(x;n)}P'_{n+1}(x)+{\bf M_{0,3}(x;n)}P_{n+1}(x).}&&\notag
\end{eqnarray}
\begin{eqnarray}
&&{\bf G_{0,3}(x;0)}= 2 (2 \alpha -3) \rho \Big( {\bf x^3}\left(4 \alpha ^2-1\right) +{\bf x^2} \left(\left(-4 \alpha
   ^2+5 \alpha -1\right) \lambda +8 \alpha -3\right)\notag\\
   && +{\bf x} \left(\left(-4 \alpha ^2+10
   \alpha -6\right) \rho +(8-8 \alpha ) \lambda +4\right)+(6-4 \alpha ) \rho -4 \lambda \Big),\notag\\
&& {\bf G_{0,3}(x;n)} =  \frac{(-2 \alpha +n+2)}{(-2 \alpha +2 n+1) (-\alpha+n+1)^2}\Big(
 -2 {\bf x^3} (2\alpha -1) n (-2 \alpha +n-1)\notag\\
 &&+\frac{2 n {\bf x^2} }{ (-\alpha +n+1)} \Big(n^2
   ((2 \alpha -2) \lambda -2)+n \left(\left(-6 \alpha ^2+9 \alpha -3\right) \lambda +8
   \alpha -4\right)\notag\\
   &&+\left(4 \alpha ^3-9 \alpha ^2+6 \alpha -1\right) \lambda -8 \alpha ^2+11 \alpha -3\Big)\notag\\
   &&+\frac{2 n {\bf x} (-2 \alpha +n+2) \left(\left(-2 \alpha ^2+5 \alpha -3\right) \rho +(4-4 \alpha ) \lambda +n ((2 \alpha -3)
   \rho +2 \lambda )+2\right)}{(-\alpha +n+1)}\notag\\
    &&-\frac{2 n (-2 \alpha +n+2) (2
   \alpha  \rho +2 \lambda -3 \rho )}{ (-\alpha +n+1)}\Big),\ n\geq 1,\notag\\
&& {\bf G_{1,3}(x;n)}= {\bf x^4}(2 \alpha -1) (n+2) (n+3) - {\bf x^3}2 (n+2)  \Big(\left(-\alpha ^2+2 \alpha -1\right) \lambda +\alpha \notag\\
   && +n^2 ((\alpha -1) \lambda -1)+n
   \left(\left(-\alpha ^2+3 \alpha -2\right) \lambda -\alpha -1\right)-1\Big)/(-\alpha +n+1)\notag\\
   && -n {\bf x^2} \Big(n^3 ((2 \alpha -3) \rho +2 \lambda )+n^2
   \left(\left(-4 \alpha ^2+12 \alpha -9\right) \rho +2 \lambda -4\right)\notag\\
   &&+n
   \left(\left(-2 \alpha ^2+4 \alpha -2\right) \lambda +\left(2 \alpha ^3-11 \alpha
   ^2+18 \alpha -9\right) \rho -2 \alpha -5\right) \notag\\
   &&+\left(-2 \alpha ^2+4 \alpha -2\right) \lambda +8 \alpha ^2+\left(2 \alpha ^3-7 \alpha
   ^2+8 \alpha -3\right) \rho -8 \alpha\Big)/(\alpha -n-1)^2\notag\\
   &&+\Big(2 n {\bf x} \left(\left(4 \alpha ^2-8 \alpha +4\right) \lambda -4 \alpha +n^2 ((3-2 \alpha ) \rho -2
   \lambda )+n \left(\left(2 \alpha ^2-5 \alpha +3\right) \rho \right.\right. \notag\\
   &&\left.\left.+(1-\alpha ) \lambda +3\right)+4\right)\Big)/(\alpha -n-1)^2- \frac{n (-4 \alpha +3
   n+4) (2 \alpha  \rho +2 \lambda -3 \rho )}{(-\alpha +n+1)^2},\notag\\
&& {\bf H_{3}(x;n)}= \frac{n^2 (-2 \alpha +n+2)}{(-2 \alpha +2 n+1) (-\alpha +n+1)^2}\notag\\
&&\hspace{1.85cm}\Big((n-1){\bf x^2}+\frac{2n{\bf x}}{(-\alpha +n+1)}+\frac{n}{(-\alpha +n+1)^2} \Big),\quad
 {\bf \Phi^3(x) }= {\bf x^6},\notag\\
&& {\bf M_{2,3}(x;n)} = {\bf x^5} (2 \alpha -n+4)+\frac{{\bf x^4} (-2 \alpha +n+2)}{(-\alpha +n+1)},\notag\\
&& {\bf M_{1,3}(x;n)} = -2 {\bf x^4} (-4 \alpha +2 n+1)+\frac{2 {\bf x^3} (-2 \alpha +n+2)}{-\alpha +n+1}-
\frac{n {\bf x^2 }(-2\alpha +n+2)}{(-\alpha +n+1)^2},\notag\\
&& {\bf M_{0,3}(x;n)}=(-2 \alpha +n+2)\left(-2{\bf x^3}+\frac{n^2 {\bf x}}{(-\alpha +n+1)^2}+\frac{n^2}{(-\alpha +n+1)^3}
 \right) .\notag
\end{eqnarray}

%%%%%%%%%%%%%%%%%%%%%%%%%%%%%%%%%%%%%%%%%%%%%%%
\noindent {\bf Fourth structure relation}
%%%%%%%%%%%%%%%%%%%%%%%%%%%%%%%%%%%%%%%%%%%%%%%
\begin{eqnarray}
{\bf {\bf G_{0,4}(x;n)}P^{(1)}_{n-1}(x)+{\bf G_{1,4}(x;n)}P^{(1)}_{n}(x)+{\bf H_{4}(x;n)} P_{n}(x)=}&& \notag\\
{\bf {\bf\Phi^4(x)}P^{(4)}_{n+1}(x)+{\bf M_{3,4}(x;n)}P^{(3)}_{n+1}(x)+{\bf M_{2,4}(x;n)}P''_{n+1}(x)}&&\notag\\
{\bf +{\bf M_{1,4}(x;n)}P'_{n+1}(x)+{\bf M_{0,4}(x;n)}P_{n+1}(x).}&&\notag
\end{eqnarray}
\begin{eqnarray}
&&{\bf  G_{0,4}(x;0)} = 4 (2 \alpha -3) \rho\notag\\
&&\Big(2{\bf x^4} (2 \alpha -1) \left(\alpha ^2+\alpha
   +1\right) +{\bf x^3 }\left(12 \alpha ^2+\left(-4 \alpha ^3+5 \alpha ^2-2 \alpha +1\right) \lambda -3 \alpha \right)\notag\\
   &&+{\bf x^2 }\left(\left(-12 \alpha ^2+17 \alpha -5\right) \lambda +\left(-4 \alpha ^3+12 \alpha ^2-11
   \alpha +3\right) \rho +12 \alpha -5\right)\notag\\
&&  -4 {\bf x} \left(\left(2 \alpha ^2-5 \alpha +3\right) \rho +(3 \alpha -3) \lambda -1\right)-2 (2 \alpha  \rho +2 \lambda -3 \rho )\Big),\notag\\
%\end{eqnarray}
%%%%%%%%%%%%%%%%%%%%%%%%%%%%%%%%%%%%%%%%%%%%%%%%%%%%%%%
%\begin{eqnarray}
&&{\bf G_{0,4}(x;n)} =  \frac{1}{(-2 \alpha +2 n+1) (-\alpha +n+1)^2}\notag\\
&&\Big(4{\bf x^4 }
   (2 \alpha -1) n (-2 \alpha +n+2) \left(2 \alpha ^2+2 \alpha +n^2-2 \alpha 
   n+2\right)\notag\\
&&-\Big(2 n {\bf x^3} (-2 \alpha +n+2) \Big(4 n^3 (\alpha  \lambda -\lambda -1)-(12 \alpha
   -7) n^2 (\alpha  \lambda -\lambda -1)\notag\\
   &&+n \left(16 \alpha ^3 \lambda -29 \alpha ^2
   \lambda -32 \alpha ^2+18 \alpha  \lambda +23 \alpha -5 \lambda -6\right)\notag\\
   &&-2 (\alpha -1) \left(4 \alpha ^3 \lambda -5 \alpha ^2 \lambda -12 \alpha ^2+2 \alpha \lambda +3 \alpha -\lambda \right)\Big)/(-\alpha +n+1)\notag\\
   &&-\Big(2 n {\bf x^2} (-2 \alpha +n+2)
   \Big(2 n^4 (2\alpha  \rho +2 \lambda -3 \rho )-(8 \alpha -7) n^3 (2 \alpha  \rho +2 \lambda -3
   \rho )\notag\\
   &&+n^2 \left(28 \alpha ^3 \rho +44 \alpha ^2 \lambda -90 \alpha ^2 \rho -73
   \alpha  \lambda +92 \alpha  \rho -22 \alpha +29 \lambda -30 \rho +12\right)\notag\\
   &&\left. -(\alpha
   -1) n \left(24 \alpha ^3 \rho +56 \alpha ^2 \lambda -74 \alpha ^2 \rho -85 \alpha 
   \lambda +71 \alpha  \rho -44 \alpha +29 \lambda -21 \rho +21\right)\right.\notag\\
   && +2 (\alpha -1)^2 \left(4 \alpha ^3 \rho +12 \alpha ^2 \lambda -12 \alpha ^2 \rho -17
   \alpha  \lambda +11 \alpha  \rho -12 \alpha +5 \lambda -3 \rho +5\right)\Big)\notag\\
   &&/(-\alpha +n+1)^2\notag\\
&&  -\Big(4 n {\bf x} (-2 \alpha +n+2)
   \left(n^2 \left(\left(8 \alpha ^2-20 \alpha +12\right) \rho +(11 \alpha -11) \lambda
   -3\right)\right.\notag\\
   &&\left. +n \left(\left(-22 \alpha ^2+44 \alpha -22\right) \lambda +\left(-16 \alpha
   ^3+56 \alpha ^2-64 \alpha +24\right) \rho +6 \alpha -6\right)\right.\notag\\
   &&\left. -4 \alpha ^2+\left(12 \alpha ^3-36 \alpha ^2+36 \alpha -12\right) \lambda +\left(8
   \alpha ^4-36 \alpha ^3+60 \alpha ^2-44 \alpha +12\right) \rho \right.\notag\\
     &&\left. +8 \alpha -4\right)\Big)/(-\alpha
   +n+1)^2- \Big(2 n (-2 \alpha +n+2) \left(4 (\alpha -1)^2+3 n^2 \right.\notag\\
     &&\left. -6 (\alpha -1) n\right)
   (2 \alpha  \rho +2 \lambda -3 \rho )\Big)/ (-\alpha +n+1)^2\Big),\ n\geq 1;\notag
   \end{eqnarray}
     %%%%%%%%%%%%%%%%%%%%%%%%%%%%%%%%%%%%%%%%%%%%%%%%%%%%%%%
\begin{eqnarray}
&&{\bf G_{1,4}(x;n)} = {\bf x^5}(2 \alpha -1) (n+2) (n+3) (n+4) \notag\\
   &&-\Big({\bf x^4} (n+2) (n+3) \left(n^4 (2 \alpha  \rho +2 \lambda -3 \rho )\right.\notag\\
   &&+n^3
   \left(-4 \alpha ^2 \rho +2 \alpha  \lambda +16 \alpha  \rho +4 \lambda -15 \rho
   -6\right)\notag\\
   &&+n^2 \left(2 \alpha ^3 \rho -4 \alpha ^2 \lambda -19 \alpha ^2 \rho +14
   \alpha  \lambda +42 \alpha  \rho -6 \lambda -27 \rho -21\right)\notag\\
   &&+ n \left(6 \alpha ^3
   \rho -12 \alpha ^2 \lambda -29 \alpha ^2 \rho -8 \alpha ^2+28 \alpha  \lambda +44
   \alpha  \rho -2 \alpha -16 \lambda -21 \rho -18\right)\notag\\
   &&-2 (\alpha -1) (\alpha  \lambda -\lambda -1)\notag\\
   &&\left. +2 n^2 (\alpha  \lambda -\lambda -1) +n
   \left(-2 \alpha ^2 \lambda +6 \alpha  \lambda -4 \alpha -4 \lambda -1\right)\right)\Big)/(-\alpha +n+1)\notag\\
   &&-2n {\bf x^3} (\alpha -1) \left(2 \alpha ^2 \rho +8 \alpha ^2-4 \alpha  \lambda -5 \alpha  \rho +8
   \alpha +4 \lambda +3 \rho \right) /(-\alpha +n+1)^2\notag\\
   &&-\Big(n {\bf x^2} \left(3 n^4 (2 \alpha  \rho +2 \lambda -3 \rho )-3 n^3 \left(4 \alpha ^2 \rho +2
   \alpha  \lambda -12 \alpha  \rho -4 \lambda +9 \rho +2\right)\right.\notag\\
   &&\left. +n^2 \left(6 \alpha ^3
   \rho +8 \alpha ^2 \lambda -33 \alpha ^2 \rho -12 \alpha  \lambda +54 \alpha  \rho
   -12 \alpha +4 \lambda -27 \rho -9\right)\right.\notag\\
   &&\left. -3n (\alpha -1)  \left(8 \alpha ^2 \lambda
   -2 \alpha ^2 \rho -6 \alpha  \lambda +5 \alpha  \rho -16 \alpha -2 \lambda -3 \rho
   \right)\right.\notag\\
   &&\left. +4 (\alpha -1)^2 \left(4 \alpha ^2 \lambda -3 \alpha  \lambda -8 \alpha -\lambda
   +1\right)\right)\Big)/(-\alpha +n+1)^3\notag\\
   &&-\Big(n {\bf x} \left(3 n^3 (2 \alpha  \rho +2 \lambda -3 \rho )+n^2 \left(2 \alpha ^2 \rho
   +12 \alpha  \lambda -5 \alpha  \rho -12 \lambda +3 \rho -10\right)\right.\notag\\
   &&\left. -12 (\alpha -1) n
   \left(2 \alpha ^2 \rho +4 \alpha  \lambda -5 \alpha  \rho -4 \lambda +3 \rho
   -2\right)\right.\notag\\
   &&\left.+8 (\alpha -1)^2 \left(2 \alpha ^2 \rho +4 \alpha  \lambda -5 \alpha  \rho -4 \lambda +3
   \rho -2\right)\right)\Big)/(-\alpha +n+1)^3\notag\\
   &&-n \left(8 \alpha ^2-16 \alpha +5 n^2-12 \alpha  n+12 n+8\right) (2 \alpha  \rho
   +2 \lambda -3 \rho )/(-\alpha +n+1)^3,\notag\\
%\end{eqnarray}
%\begin{eqnarray}
&&{\bf H_{4}(x;n)} =\frac{1}{(-2 \alpha +2 n+1) (-\alpha +n+1)^2}\Big( -{\bf x^3}(n-2) (n-1) n^2  (-2 \alpha
   +n+2)\notag\\
   &&-\frac{3 {\bf x^2}  (n-1) n^3(-2 \alpha+n+2)}{(-\alpha +n+1)}-\frac{3 {\bf x}n^4  (-2\alpha +n+2)}{(-\alpha +n+1)^2}-\frac{n^4 (-2 \alpha +n+2)}{(-\alpha +n+1)^3}\Big),\notag\\
&&  {\bf \Phi^4(x)} =  {\bf x^8},\notag\\
&& {\bf M_{3,4}(x;n)}= {\bf x^7} (2 \alpha -n+10)+\frac{{\bf x^6 }(-2 \alpha +n+2)}{(-\alpha +n+1)},\notag\\
&&{\bf  M_{2,4}(x;n)}= -9 {\bf x^6 }(-2 \alpha +n-2)+\frac{(-2 \alpha +n+2)}{(-\alpha +n+1)}\left(6 {\bf x^5}-
\frac{n {\bf x^4 }}{(-\alpha +n+1)}\right),\notag\\
&&{\bf M_{1,4}(x;n)}= -6 {\bf x^5} (-6 \alpha +3 n+2)+\frac{(-2 \alpha +n+2)}{(-\alpha +n+1)}
\notag\\
&&\hspace{2,0cm}\left(6 {\bf x^4}+
\frac{(n-2) n {\bf x^3}}{(-\alpha+n+1)}+\frac{n^2 {\bf x^2} }{(-\alpha +n+1)^2}\right),\notag\\
&&{\bf M_{0,4}(x;n)}= -6 {\bf x^4 }(-2 \alpha +n+2)-\frac{(-2 \alpha +n+2)}{(-\alpha +n+1)^2}\notag\\
&&\hspace{2,0cm}\left(-(n-2) n^2 {\bf x^2}-\frac{2 n^3 {\bf x} }{(-\alpha +n+1)}+\frac{n^3}{(-\alpha +n+1)^2}\right).\notag
\end{eqnarray}

%%%%%%%%%%%%%%%%%%%%%%%%%%%%%%%%%%%
%%%%%%%%%%%%%%%%%%%%%%%%%%%%%%%%%%%
\subsection{The fourth-order linear differential equation}
%%%%%%%%%%%%%%%%%%%%%%%%%%%%%%%%%%%
%%%%%%%%%%%%%%%%%%%%%%%%%%%%%%%%%%%
\begin{eqnarray}
{\bf \mathcal{A}(x;n)P^{(4)}_{n+1}(x)
+\mathcal{B}(x;n)P^{(3)}_{n+1}(x)
+\mathcal{C}(x;n)P''_{n+1}(x)
+\mathcal{D}(x;n)P'_{n+1}(x)}&&\notag\\
{\bf +\mathcal{E}(x;n)P_{n+1}(x)=0,\; n\geq 0.}&&\notag
\end{eqnarray}
\quad\\
\noindent Let $c(x;n)$ be the greatest common factor between $\mathcal{A}$, $\mathcal{B}$, $\mathcal{C}$, $\mathcal{D}$, and $\mathcal{E}$. Then, we introduce the notation
$${\bf \widehat{\mathcal{A}}=\mathcal{A}/c\quad ,\quad \widehat{\mathcal{B}}=\mathcal{B}/c\quad ,\quad \widehat{\mathcal{C}}=\mathcal{C}/c\quad ,\quad 
\widehat{\mathcal{D}}=\mathcal{D}/c\quad ,\quad \widehat{\mathcal{E}}=\mathcal{E}/c.}$$
\noindent Greatest common factor between $\mathcal{A}$, $\mathcal{B}$, $\mathcal{C}$, $\mathcal{D}$, and $\mathcal{E}$: 
\begin{eqnarray}
{\bf  c(x;0)}=4 (3-2 \alpha )^2 \rho ^2 {\bf x^6}\ ,\  {\bf c(x;n)}=\frac{4 n^2 {\bf x^6} (-2 \alpha +n+2)^2}{(-2 \alpha +2 n+1)^2 (-\alpha +n+1)^4},\ n\geq 1. \notag
\end{eqnarray}
%%%%%%%%%%%%%%%%%%%%%%%%%%%%%%%%%%%%%%%%%%%%%%%
\begin{eqnarray}
&&{\bf \widehat{\mathcal{A}}(x;n)}= {\bf x^8} (2 \alpha -1)^2 (n+1) (-2 \alpha +n+2) +2 {\bf x^7}(2 \alpha -1)  \left(\left(3 \alpha ^2-7 \alpha +4\right) \lambda \right.\notag\\
&&\left. -5 \alpha +n^2 ((2-2 \alpha ) \lambda
   +2)+n \left(\left(4 \alpha ^2-10 \alpha +6\right) \lambda -4 \alpha +6\right)+5\right)\notag\\
   &&+{\bf x^6 }\left(\left(28 \alpha ^2-60 \alpha +26\right) \lambda +\left(-4 \alpha ^3+13 \alpha ^2-14
   \alpha +5\right) \lambda ^2\right.\notag\\
   &&+\left(8 \alpha ^3-36 \alpha ^2+46 \alpha -15\right) \rho  -16 \alpha \notag\\
   &&\left.
  + n^2 \left(\left(4 \alpha ^2-8 \alpha +4\right) \lambda ^2+\left(-8
   \alpha ^2+16 \alpha -6\right) \rho +(12-16 \alpha ) \lambda +4\right)\right.\notag\\
   && +n
   \left(\left(32 \alpha ^2-72 \alpha +36\right) \lambda +\left(-8 \alpha ^3+28 \alpha
   ^2-32 \alpha +12\right) \lambda ^2+\left(16 \alpha ^3-56 \alpha ^2 \right.\right.\notag\\
   &&\left.\left.\left. +60 \alpha
   -18\right) \rho -8 \alpha +12\right)+11\right)\notag\\
   &&+ 2 {\bf x^5} \left(\left(-4 \alpha ^2+10 \alpha -6\right) \lambda ^2+\left(6 \alpha ^2-15 \alpha +9\right)
   \rho \right.\notag\\
   && +\lambda  \left(\left(-2 \alpha ^3+9 \alpha ^2-13 \alpha +6\right) \rho +10
   \alpha -10\right)+ n^2 \left(\lambda  \left(\left(4 \alpha ^2-10 \alpha +6\right)
   \rho -4\right) \right.\notag\\
   &&\left. \left. +(4 \alpha -4) \lambda ^2+(6-4 \alpha ) \rho \right)+ n \left(\left(-8
   \alpha ^2+20 \alpha -12\right) \lambda ^2+\left(8 \alpha ^2-24 \alpha +18\right)
   \rho \right.\right.\notag\\
   &&\left.\left. +\lambda  \left(\left(-8 \alpha ^3+32 \alpha ^2-42 \alpha +18\right) \rho +8
   \alpha -12\right)\right)-2\right)\notag\\
   &&- {\bf x^4} (2 \alpha  \rho +2 \lambda -3 \rho ) \left((2 \alpha -2) \lambda \right.\notag\\
   &&\left. +n^2 ((3-2 \alpha ) \rho -2 \lambda )+n \left(\left(4 \alpha
   ^2-12 \alpha +9\right) \rho +(4 \alpha -6) \lambda \right)-2\right),\notag 
\end{eqnarray}
%%%%%%%%%%%%%%%%%%%%%%%%%%%%%%%%%%%%%%
\begin{eqnarray}
&&{\bf \widehat{\mathcal{B}}(x;n) } =  6 {\bf x^7}(2 \alpha -1)^2 (n+1)  (-2 \alpha +n+2)+ 14 {\bf x^6} (2 \alpha -1) \left(\left(3 \alpha ^2-7 \alpha +4\right) \lambda \right.\notag\\
&& \left. -5 \alpha +n^2 ((2-2 \alpha ) \lambda
   +2)+n \left(\left(4 \alpha ^2-10 \alpha +6\right) \lambda -4 \alpha +6\right)+5\right)\notag\\
   &&+8 {\bf x^5} \left(\left(28 \alpha ^2-60 \alpha +26\right) \lambda +\left(-4 \alpha ^3+13 \alpha ^2-14
   \alpha +5\right) \lambda ^2\right.\notag\\
   &&+\left(8 \alpha ^3-36 \alpha ^2+46 \alpha -15\right) \rho  -16 \alpha \notag\\
   &&\left.
  + n^2 \left(\left(4 \alpha ^2-8 \alpha +4\right) \lambda ^2+\left(-8
   \alpha ^2+16 \alpha -6\right) \rho +(12-16 \alpha ) \lambda +4\right)\right.\notag\\
  && +n \left(\left(32 \alpha ^2-72 \alpha +36\right) \lambda +\left(-8 \alpha ^3+28 \alpha
   ^2-32 \alpha +12\right) \lambda ^2 \right.\notag\\
   &&\left.\left. +\left(16 \alpha ^3-56 \alpha ^2+60 \alpha
   -18\right) \rho-8 \alpha +12\right)+11\right)\notag\\
   &&+18 {\bf x^4} \left(\left(-4 \alpha ^2+10 \alpha -6\right) \lambda ^2+\left(6 \alpha ^2-15 \alpha +9\right)
   \rho +\lambda  \left(\left(-2 \alpha ^3+9 \alpha ^2 \right.\right.\right.\notag\\
   &&\left.\left. -13 \alpha +6\right) \rho +10
   \alpha -10\right)+ n^2 \left(\lambda  \left(\left(4 \alpha ^2-10 \alpha +6\right)
   \rho -4\right)+(4 \alpha -4) \lambda ^2\right.\notag\\
   &&\left.+(6-4 \alpha ) \rho \right)
    +n \left(\left(-8
   \alpha ^2+20 \alpha -12\right) \lambda ^2+\left(8 \alpha ^2-24 \alpha +18\right)
   \rho  \right.\notag\\
   &&\left.\left. +\lambda  \left(\left(-8 \alpha ^3+32 \alpha ^2-42 \alpha +18\right) \rho+8
   \alpha -12\right)\right)-2\right) \notag\\
   &&-10 {\bf x^3 }(2 \alpha  \rho +2 \lambda
   -3 \rho ) \left((2 \alpha -2) \lambda +n^2 ((3-2 \alpha ) \rho -2 \lambda ) \right.\notag\\
   &&\left. +n \left(\left(4 \alpha
   ^2-12 \alpha +9\right) \rho +(4 \alpha -6) \lambda \right)-2\right),\notag 
\end{eqnarray}
%%%%%%%%%%%%%%%%%%%%%%%%%%%%%%%%%%%%%%
\begin{eqnarray}
&&{\bf \widehat{\mathcal{C}}(x;n) }=  -2{\bf x^6} (n+1) (n-2 \alpha +2) (2 \alpha -1)^2 \left(n^2+n(-2 \alpha  +3 )+2 \alpha ^2-4 \alpha
   -1\right) \notag\\
   &&-2 {\bf x^5}(2 \alpha -1) \left(-36 \alpha ^3+96 \alpha ^2+\lambda  \left(12 \alpha ^4-56 \alpha ^3+\left(16 \alpha
   ^2-40 \alpha +24\right) \rho \right.\right.\notag\\
   &&\left. +71 \alpha ^2-13 \alpha -30\right)+(16 \alpha -16)
   \lambda ^2+ (24-16 \alpha ) \rho -45 \alpha +n^4 ((4-4 \alpha ) \lambda +4) \notag\\
   && +n^3
   \left(\left(16 \alpha ^2-40 \alpha +24\right) \lambda -16 \alpha +24\right)+ n^2
   \left(\left(-16 \alpha ^2+48 \alpha -36\right) \rho ^2+32 \alpha ^2 \right.\notag\\
   &&\left. +\lambda 
   \left(-24 \alpha ^3+96 \alpha ^2+(48-32 \alpha ) \rho -109 \alpha +37\right)-84
   \alpha -16 \lambda ^2+41\right)+\notag\\
  && n \left(-32 \alpha ^3+120 \alpha ^2+\left(32 \alpha
   ^3-144 \alpha ^2+216 \alpha -108\right) \rho ^2+\lambda  \left(16 \alpha ^4-88
   \alpha ^3 \right.\right.\notag\\
   &&\left.\left. +\left(64 \alpha ^2-192 \alpha +144\right) \rho +146 \alpha ^2-77 \alpha
   +3\right)+\left.(32 \alpha -48) \lambda ^2-118 \alpha +15\right)-15\right) \notag\\
   %%%%%%%%%%%%%%%%%%%%%%%%%
   &&-2{\bf x^4 }\left(-128 \alpha ^3+298 \alpha ^2+\left(104 \alpha ^4-432 \alpha ^3+488 \alpha ^2-63 \alpha
   -61\right) \lambda +\left(-8 \alpha ^5+46 \alpha ^4 \right.\right.\notag\\
   &&\left. -77 \alpha ^3+28 \alpha ^2+31
   \alpha -20\right) \lambda ^2+\left(16 \alpha ^5-120 \alpha ^4+292 \alpha ^3-234
   \alpha ^2-2 \alpha +30\right) \rho  \notag\\
   &&\left. -127 \alpha+n^4 \left(\left(4 \alpha ^2-8 \alpha
   +4\right) \lambda ^2+\left(-8 \alpha ^2+16 \alpha -6\right) \rho +(12-16 \alpha )
   \lambda +4\right)\right.\notag\\
   && +n^3 \left(\left(64 \alpha ^2-144 \alpha +72\right) \lambda
   +\left(-16 \alpha ^3+56 \alpha ^2-64 \alpha +24\right) \lambda ^2+\left(32 \alpha
   ^3-112 \alpha ^2 \right.\right.\notag\\
   &&\left.\left.\left. +120 \alpha -36\right) \rho -16 \alpha +24\right)+n^2 \left(64
   \alpha ^2+\left(-128 \alpha ^3+440 \alpha ^2-410 \alpha +106\right) \lambda \right.\right.\notag\\
   && +\left(24 \alpha ^4-120 \alpha ^3+193 \alpha ^2-122 \alpha +25\right) \lambda
   ^2\notag\\
   &&+\left(-48 \alpha ^4+240 \alpha ^3-404 \alpha ^2+268 \alpha -60\right) \rho \notag\\
   &&\left. -128
   \alpha +43\right)+ n \left(-96 \alpha ^3+304 \alpha ^2+\left(128 \alpha ^4-592 \alpha
   ^3+844 \alpha ^2-362 \alpha -6\right) \lambda \right.\notag\\
   && +\left(-16 \alpha ^5+104 \alpha ^4-210
   \alpha ^3+139 \alpha ^2 +16 \alpha -33\right) \lambda ^2+\left(32 \alpha ^5-208
   \alpha ^4+472 \alpha ^3 \right. \notag\\
   &&\left.\left.\left. -452 \alpha ^2+168 \alpha -18\right) \rho  -254 \alpha
   +21\right)-13\right)\notag\\
   &&- 2 {\bf x^3}\left(-112 \alpha ^2+\left(-32 \alpha ^4+148 \alpha ^3 -186 \alpha ^2 +20 \alpha +50\right)
   \lambda ^2\right. \notag\\
   &&+\left(56 \alpha ^4-308 \alpha ^3+484 \alpha ^2-202 \alpha -30\right) \rho +\lambda  \left(176 \alpha ^3-548 \alpha ^2 \right.\notag\\
   &&\left. +\left(-8 \alpha ^5+60 \alpha ^4-140
   \alpha ^3+90 \alpha ^2+58 \alpha -60\right) \rho +364 \alpha +8\right)+168 \alpha \notag\\
   &&\left. +n^4 \left(\lambda  \left(\left(8 \alpha ^2-20 \alpha +12\right) \rho -8\right)+(8
   \alpha -8) \lambda ^2+(12-8 \alpha ) \rho \right)\right.\notag\\
   &&+n^3 \left(\left(-32 \alpha ^2+80 \alpha -48\right) \lambda ^2+
   \left(32 \alpha ^2-96 \alpha +72\right) \rho \right.\notag\\
   &&\left.+\lambda  \left(\left(-32 \alpha ^3+128 \alpha ^2-168 \alpha +72\right) \rho+32 \alpha
   -48\right)\right) \notag\\
   &&+n^2 \left(\left(64 \alpha ^3-240 \alpha ^2+230 \alpha -54\right)
   \lambda ^2+\left(-80 \alpha ^3+312 \alpha ^2-342 \alpha +81\right) \rho \right.\notag\\
   &&\left.\left. +\lambda 
   \left(-128 \alpha ^2+\left(48 \alpha ^4-264 \alpha ^3+470 \alpha ^2-311 \alpha
   +57\right) \rho +280 \alpha -94\right)+32 \alpha -24\right)\right.\notag\\
   && +n \left(-64 \alpha
   ^2+\left(-64 \alpha ^4+320 \alpha ^3-460 \alpha ^2+150 \alpha +54\right) \lambda
   ^2+\left(96 \alpha ^4-480 \alpha ^3+756 \alpha ^2 \right.\right.\notag\\
   &&\left. -324 \alpha -81\right) \rho \notag\\
   &&+\lambda  \left(192 \alpha ^3-656 \alpha ^2+\left(-32 \alpha ^5+224 \alpha ^4-484
   \alpha ^3+304 \alpha ^2+141 \alpha -153\right) \rho\right.\notag\\
   &&\left.\left.\left.  +596 \alpha -66\right)+144
   \alpha -72\right)-26\right) \notag\\
   %%%%%%%%%%%%%%%%%%%%%%%
   &&- 2 {\bf x^2}\left(\left(-48 \alpha ^3+166 \alpha ^2-140 \alpha +22\right) \lambda ^2+\left(72 \alpha
   ^3-276 \alpha ^2 +288 \alpha -54\right) \rho\right.\notag\\
   && +\lambda  \left(144 \alpha ^2+\left(-24
   \alpha ^4+132 \alpha ^3-216 \alpha ^2+108 \alpha \right) \rho -296 \alpha
   +92\right)-48 \alpha \notag\\
   && +n^4 \left(\left(4 \alpha ^2-12 \alpha +9\right) \rho ^2+(8
   \alpha -12) \lambda  \rho +4 \lambda ^2\right)+ n^3 \left(\left(-32 \alpha ^2+96
   \alpha -72\right) \lambda  \rho \right.\notag\\
   &&\left. +\left(-16 \alpha ^3+72 \alpha ^2-108 \alpha
   +54\right) \rho ^2+(24-16 \alpha ) \lambda ^2\right)+n^2 \left(\left(64 \alpha
   ^2-152 \alpha +44\right) \lambda ^2 \right.\notag\\
   && +\left(-48 \alpha ^2+112 \alpha -60\right) \rho
   +\lambda  \left(\left(80 \alpha ^3-328 \alpha ^2+352 \alpha -60\right) \rho -64
   \alpha +56\right)\notag\\
   &&\left. +\left(24 \alpha ^4-144 \alpha ^3+274 \alpha ^2-174 \alpha
   +9\right) \rho ^2+8\right)\notag\\
   &&+n \left(\left(-96 \alpha ^3+352 \alpha ^2-328 \alpha
   +24\right) \lambda ^2 +\left(96 \alpha ^3-368 \alpha ^2+456 \alpha -180\right) \rho \right.\notag\\
   &&+\lambda  \left(128 \alpha ^2+\left(-96 \alpha ^4+512 \alpha ^3-824 \alpha ^2+312 \alpha +144\right) \rho-304 \alpha +168\right) \notag\\
   &&\left.\left. +\left(-16 \alpha ^5+120 \alpha
   ^4-260 \alpha ^3+90 \alpha ^2+270 \alpha -216\right) \rho ^2-16 \alpha +24\right)+34\right) \notag\\
   %%%%%%%%%%%%%%%%%%%%%%%%%%%%
    &&-8{\bf  x}\left(\left(-8 \alpha ^2+18 \alpha -10\right) \lambda ^2+\left(10 \alpha ^2-25 \alpha
   +15\right) \rho \right.\notag\\
   && +\lambda  \left(\left(-6 \alpha ^3+23 \alpha ^2-29 \alpha +12\right)
   \rho + 14 \alpha -14\right)\notag\\
   &&+n^2 \left(\lambda  \left(\left(12 \alpha ^2-30 \alpha
   +18\right) \rho -4\right)+\left(4 \alpha ^3-16 \alpha ^2+21 \alpha -9\right) \rho
   ^2+(8 \alpha -8) \lambda ^2\right.\notag\\
   &&\left. +(6-4 \alpha ) \rho \right)+n \left(\left(-16 \alpha
   ^2+40 \alpha -24\right) \lambda ^2+\left(8 \alpha ^2-24 \alpha +18\right) \rho\right.\notag\\
   &&+\lambda  \left(\left(-24 \alpha ^3+96 \alpha ^2 -126 \alpha +54\right) \rho +8
   \alpha -12\right)+\notag\\
   &&\left.\left.\left(-8 \alpha ^4+44 \alpha ^3-90 \alpha ^2+81 \alpha -27\right)
   \rho ^2\right)-2\right)  \notag\\
   && +4 (2 \lambda +2 \alpha  \rho -3 \rho ) \left((2 \alpha -2) \lambda +n^2 ((3-2 \alpha ) \rho -2 \lambda )\right.\notag\\
   &&\left.+n \left(\left(4 \alpha
   ^2-12 \alpha +9\right) \rho +(4 \alpha -6) \lambda \right)-2\right),\notag 
\end{eqnarray}
%%%%%%%%%%%%%%%%%%%%%%%%%%%%%%%%%%%%%%
\begin{eqnarray}
&&{\bf \widehat{\mathcal{D}}(x;n)} =   -2 {\bf x^5}(n+1) (n-2 \alpha +2) (2 \alpha -1)^2 \left(n^2-+n(2 \alpha  +3 )+2 \alpha ^2-4 \alpha
   +2\right) -\notag\\
   %%%%%%%%%%%%%%%%%%%%%%%%%%%%%%%
   &&2{\bf x^4} (2 \alpha -1) \left(-24 \alpha ^3+\left(32 \alpha ^2-48 \alpha +16\right) \lambda ^2+78 \alpha ^2+ \lambda 
   \left(16 \alpha ^4-70 \alpha ^3+116 \alpha ^2 \right.\right.\notag\\
   &&\left. +\left(32 \alpha ^3-96 \alpha ^2+88
   \alpha -24\right) \rho -86 \alpha +24\right)-84 \alpha +n^4 ((6-6 \alpha ) \lambda
   +6)\notag\\
   &&+n^3 \left(\left(24 \alpha ^2-60 \alpha +36\right) \lambda -24 \alpha
   +36\right)\notag\\
   &&+n^2 \left(32 \alpha ^2+\lambda  \left(-36 \alpha ^3+\left(-80 \alpha
   ^2+200 \alpha -120\right) \rho +139 \alpha ^2-181 \alpha +62\right)\right.\notag\\
   &&\left.+\left(-48 \alpha^3+192 \alpha ^2-252 \alpha +108\right) \rho ^2+(32-32 \alpha ) \lambda ^2+(24-16
   \alpha ) \rho -107 \alpha +81\right)\notag\\
   &&+n \left(-16 \alpha ^3+\left(64 \alpha ^2-160
   \alpha +96\right) \lambda ^2+\left(32 \alpha ^2-96 \alpha +72\right) \rho +94 \alpha^2 \right. \notag\\
  && +\lambda  \left(24 \alpha ^4-122 \alpha ^3+239 \alpha ^2+ \left(160 \alpha ^3-640
   \alpha ^2+840 \alpha -360\right) \rho -181 \alpha +24\right) \notag\\
  &&\left.\left. +\left(96 \alpha ^4-528
   \alpha ^3+1080 \alpha ^2-972 \alpha +324\right) \rho ^2-159 \alpha +81\right)+30\right)\notag\\
   %%%%%%%%%%%%%%%%%%%%%%%%%%%%%%%
   &&-4 {\bf x^3}\left(-24 \alpha ^3+70 \alpha ^2  +\left(48 \alpha ^4-208 \alpha ^3+328 \alpha ^2-218 \alpha
   +50\right) \lambda \right.\notag\\
   &&+\left(-8 \alpha ^5+46 \alpha ^4-101 \alpha ^3+106 \alpha ^2-53
   \alpha +10\right) \lambda ^2\notag\\
   &&+\left(8 \alpha ^5-60 \alpha ^4+170 \alpha ^3-225 \alpha
   ^2+137 \alpha -30\right) \rho -63 \alpha +n^4 \left(\left(4 \alpha ^2-8 \alpha
   +4\right) \lambda ^2 \right.\notag\\
   &&\left. +\left(-8 \alpha ^2+16 \alpha -6\right) \rho +(12-16 \alpha )
   \lambda +4\right)+n^3 \left(\left(64 \alpha ^2-144 \alpha +72\right) \lambda \right.\notag\\
   &&\left. +\left(-16 \alpha ^3+56 \alpha ^2-64 \alpha +24\right) \lambda ^2+\left(32 \alpha
   ^3-112 \alpha ^2+120 \alpha -36\right) \rho -16 \alpha +24\right)\notag\\
   && +n^2 \left(16
   \alpha ^2+\left(-96 \alpha ^3+350 \alpha ^2-427 \alpha +157\right) \lambda +\left(24
   \alpha ^4-120 \alpha ^3+217 \alpha ^2 \right.\right.\notag\\
   &&\left.\left. -170 \alpha +49\right) \lambda ^2 +\left(-48
   \alpha ^4+220 \alpha ^3-392 \alpha ^2+309 \alpha -81\right) \rho -74 \alpha
   +52\right) \notag\\
   && +n \left(52 \alpha ^2+\left(64 \alpha ^4-316 \alpha ^3+608 \alpha ^2-515
   \alpha +147\right) \lambda  \right.\notag\\
   &&+\left(-16 \alpha ^5+104 \alpha ^4 -258 \alpha ^3+307
   \alpha ^2-176 \alpha +39\right) \lambda ^2 \notag\\
   &&\left.\left. +\left(32 \alpha ^5-168 \alpha ^4+388
   \alpha ^3-498 \alpha ^2 +333 \alpha-81\right) \rho -110 \alpha +48\right)+17\right)\notag\\
   %%%%%%%%%%%%%%%%%%%%%%%%%%%%%%%%%
   && -2 {\bf x^2}\left(-16 \alpha ^2+\left(-48 \alpha ^4+236 \alpha ^3-410 \alpha ^2+292 \alpha 
   -70\right)\lambda ^2\right.\notag\\
   && +\left(32 \alpha ^4-176 \alpha ^3+328 \alpha ^2-244 \alpha +60\right) \rho +\lambda  \left(96 \alpha ^3-308 \alpha ^2 \right.\notag\\
   &&\left. +\left(-16 \alpha ^5+120 \alpha ^4-340
   \alpha ^3+450 \alpha ^2-274 \alpha +60\right) \rho +294 \alpha -82\right)+24 \alpha
   \notag\\
   &&\left. +n^4 \left(\lambda  \left(\left(20 \alpha ^2-50 \alpha +30\right) \rho
   -20\right)+(20 \alpha -20) \lambda ^2+(30-20 \alpha ) \rho \right)\right.\notag\\
   && +n^3
   \left(\left(-80 \alpha ^2+200 \alpha -120\right) \lambda ^2+\left(80 \alpha ^2-240
   \alpha +180\right) \rho \right. \notag\\
   &&\left. +\lambda  \left(\left(-80 \alpha ^3+320 \alpha ^2-420 \alpha
   +180\right) \rho +80 \alpha -120\right)\right)\notag\\
   &&\left. +n^2 \left(\left(128 \alpha ^3-504
   \alpha ^2+634 \alpha -258\right) \lambda ^2+\left(-136 \alpha ^3+558 \alpha ^2-769
   \alpha +357\right) \rho \right.\right.\notag\\
   &&\left.\left. +\lambda  \left(-128 \alpha ^2+\left(120 \alpha ^4-650
   \alpha ^3+1295 \alpha ^2-1125 \alpha +360\right) \rho +370 \alpha -262\right)-16
   \alpha -6\right)\right.\notag\\
   &&\left. +n \left(32 \alpha ^2+\left(-96 \alpha ^4+512 \alpha ^3-980 \alpha
   ^2+798 \alpha -234\right) \lambda ^2+\left(112 \alpha ^4-564 \alpha ^3+1052 \alpha
   ^2 \right.\right.\right.\notag\\
   &&\left. -861 \alpha +261\right) \rho +\lambda  \left(96 \alpha ^3-404 \alpha ^2+\left(-80
   \alpha ^5+540 \alpha ^4-1420 \alpha ^3+1815 \alpha ^2-1125 \alpha \right.\right.\notag\\
   &&\left.\left.\left.\left.+270\right) \rho
   +554 \alpha -246\right)-36 \alpha -18\right)-8\right)\notag\\
   %%%%%%%%%%%%%%%%%%%%%%%%%%%%%%%%%
  && - 2 {\bf x}\left(\left(-48 \alpha ^3+168 \alpha ^2-168 \alpha +48\right) \lambda ^2 +\left(16 \alpha
   ^3-48 \alpha ^2+44 \alpha -12\right) \rho  \right.\notag\\
   && +\lambda \left(32 \alpha ^2+\left(-32
   \alpha ^4+176 \alpha ^3-328 \alpha ^2+244 \alpha -60\right) \rho-48 \alpha
   +16\right) \notag\\
   && +n^4 \left(\left(12 \alpha ^2-36 \alpha +27\right) \rho ^2+(24 \alpha -36)
   \lambda  \rho +12 \lambda ^2\right)\notag\\
   &&+n^3 \left(\left(-96 \alpha ^2+288 \alpha
   -216\right) \lambda  \rho +\left(-48 \alpha ^3+216 \alpha ^2-324 \alpha +162\right)
   \rho ^2\right.\notag\\
   &&\left.+(72-48 \alpha ) \lambda ^2\right)\notag\\
   && +n^2 \left(\left(96 \alpha ^2-264 \alpha
   +156\right) \lambda ^2+\left(-16 \alpha ^2+48 \alpha -36\right) \rho +\lambda 
   \left(\left(176 \alpha ^3-760 \alpha ^2 \right.\right.\right.\notag\\
   &&\left.\left.\left. +1040 \alpha -444\right) \rho
   +8\right)+\left(72 \alpha ^4-432 \alpha ^3+942 \alpha ^2-882 \alpha +297\right) \rho
   ^2-8\right)\notag\\
   &&\left. +n \left(\left(-96 \alpha ^3+384 \alpha ^2-456 \alpha +144\right) \lambda
   ^2+\left(32 \alpha ^3-144 \alpha ^2+216 \alpha -108\right) \rho \right.\right.\notag\\
   &&\left.\left. +\lambda 
   \left(\left(-160 \alpha ^4+896 \alpha ^3-1768 \alpha ^2+1416 \alpha -360\right) \rho
   -16 \alpha +24\right)\right.\right.\notag\\
   &&\left.\left. +\left(-48 \alpha ^5+360 \alpha ^4-1020 \alpha ^3+1350 \alpha
   ^2-810 \alpha +162\right) \rho ^2+16 \alpha -24\right)\right)
   \notag\\
   &&+4 (2 \lambda +2 \alpha  \rho -3 \rho ) \left(\left(4 \alpha ^2-6 \alpha +2\right) \lambda +n^2 \left(\left(-6 \alpha ^2+15 \alpha
   -9\right) \rho +(4-4 \alpha ) \lambda -2\right)\right.\notag\\
   &&\left. +n \left(\left(8 \alpha ^2-20 \alpha
   +12\right) \lambda +\left(12 \alpha ^3-48 \alpha ^2+63 \alpha -27\right) \rho +4
   \alpha -6\right) \right),\notag 
\end{eqnarray}
%%%%%%%%%%%%%%%%%%%%%%%%%%%%%%%%%%%%%%
\begin{eqnarray}
&&{\bf \widehat{\mathcal{E}}(x;n) } =  {\bf x^4}(n+1)^3 (n-2 \alpha +2)^3 (2 \alpha -1)^2 - 
2{\bf x^3} (n+1) (n-2 \alpha +2) (2 \alpha -1)\notag\\
   &&\left(-12 \alpha ^2 +\left(4 \alpha ^3-17 \alpha ^2+23 \alpha -10\right) \lambda +27 \alpha
   +n^4 ((2 \alpha -2) \lambda -2)\right.\notag\\
   && +n^3 \left(\left(-8 \alpha ^2+20 \alpha -12\right)
   \lambda +8 \alpha -12\right)+n^2 \left(-8 \alpha ^2+\left(8 \alpha ^3-37 \alpha
   ^2+56 \alpha -27\right) \lambda \right.\notag\\
   &&\left.\left. +35 \alpha -30\right)+n \left(-22 \alpha ^2+\left(10
   \alpha ^3-43 \alpha ^2+60 \alpha -27\right) \lambda +57 \alpha -36\right)-15\right) \notag\\
   && +{\bf x^2} (n+1)(n-2 \alpha +2)\left(48 \alpha ^2+\left(-48 \alpha ^3+188 \alpha ^2-222 \alpha +70\right)
   \lambda  \right.\notag\\
   && +\left(-16 \alpha ^4+104 \alpha ^3-236 \alpha ^2+214 \alpha-60\right) \rho -92 \alpha  \notag\\
   && +n^4 \left(\left(4 \alpha ^2-8 \alpha
   +4\right) \lambda ^2+\left(-8 \alpha ^2+16 \alpha -6\right) \rho+(12-16\alpha) \lambda +4\right) \notag\\
   &&+n^3 \left(\left(64 \alpha ^2-144 \alpha
   +72\right) \lambda +\left(-16 \alpha ^3+56 \alpha ^2-64 \alpha
   +24\right) \lambda ^2 \right.\notag\\
   &&\left. +\left(32 \alpha ^3-112 \alpha ^2+120 \alpha
   -36\right) \rho -16 \alpha +24\right)\notag\\
   &&+n^2 \left(16 \alpha ^2+\left(-64
   \alpha ^3+292 \alpha ^2 -424 \alpha +166\right) \lambda \right.\notag\\
   &&+\left(16 \alpha
   ^4-84 \alpha ^3+161 \alpha ^2-134 \alpha +41\right) \lambda ^2\notag\\
   &&\left.\left. +\left(-32
   \alpha ^4+168 \alpha ^3-364 \alpha ^2+342 \alpha -99\right) \rho -88
   \alpha +59\right)\right.\notag\\
   &&+n \left(80 \alpha ^2+\left(-104 \alpha ^3+428 \alpha
   ^2-524 \alpha +174\right) \lambda\right. \notag\\
   &&+\left(8 \alpha ^4-38 \alpha ^3+67
   \alpha ^2-52 \alpha +15\right) \lambda ^2\notag\\
   &&\left.\left. +\left(-16 \alpha ^4+176 \alpha
   ^3-480 \alpha ^2+468 \alpha -135\right) \rho -166 \alpha +69\right)+34\right) \notag\\
   && +2{\bf x}  (n+1) (n-2 \alpha +2)\left(\left(-24 \alpha ^2+60 \alpha -24\right) \lambda +\left(-16 \alpha ^3+72
   \alpha ^2-92 \alpha +30\right) \rho  \right.\notag\\
   && +n^4 \left(\lambda 
   \left(\left(4 \alpha ^2-10 \alpha +6\right) \rho -4\right)+(4 \alpha -4)
   \lambda ^2+(6-4 \alpha ) \rho \right)\notag\\
   &&+n^3 \left(\left(-16 \alpha ^2+40 \alpha -24\right) \lambda ^2 +\left(16 \alpha ^2-48 \alpha +36\right)
   \rho \right.\notag\\
   &&\left.+\lambda  \left(\left(-16 \alpha ^3+64 \alpha ^2-84 \alpha
   +36\right) \rho +16 \alpha -24\right)+8 \alpha\right)\notag\\
   &&\left. +n^2 \left(\left(16 \alpha
   ^3-68 \alpha ^2+96 \alpha -44\right) \lambda ^2+\left(-16 \alpha ^3+82
   \alpha ^2-131 \alpha +66\right) \rho \right.\right.\notag\\
   &&\left.\left. +\lambda  \left(-16 \alpha
   ^2+\left(16 \alpha ^4-86 \alpha ^3+177 \alpha ^2-164 \alpha +57\right)
   \rho +70 \alpha -56\right)-6\right)\right.\notag\\
   &&\left. +n \left(\left(8 \alpha ^3-36 \alpha
   ^2+52 \alpha -24\right) \lambda ^2+\left(-20 \alpha ^3+76 \alpha ^2-93
   \alpha +36\right) \rho \right.\right.\notag\\
   &&\left.\left. +\lambda  \left(-44 \alpha ^2+\left(-4 \alpha
   ^4+12 \alpha ^3-5 \alpha ^2-12 \alpha +9\right) \rho +106 \alpha
   -60\right)+12 \alpha -18\right)-4\right) \notag\\
   && +(n+1) (n-2 \alpha +2) (2 \lambda +2
   \alpha  \rho -3 \rho ) \left(-8 \alpha +n^4 ((2 \alpha -3) \rho +2 \lambda )\right.\notag\\
   && +n^3 \left(\left(-8 \alpha^2+24 \alpha -18\right) \rho +(12-8 \alpha ) \lambda \right)\notag\\
   &&+n^2 \left(\left(8 \alpha ^2-26 \alpha +16\right) \lambda +\left(8 \alpha
   ^3-32 \alpha ^2+40 \alpha -15\right) \rho +6\right)\notag\\
   &&\left. +n \left(\left(4
   \alpha ^2-2 \alpha -6\right) \lambda +\left(-8 \alpha ^3+40 \alpha ^2-66
   \alpha +36\right) \rho -12 \alpha +18\right)+4\right).\notag 
\end{eqnarray}

%%%%%%%%%%%%%%%%%%%%%%%%%%%%%%%%%%%%%%%%%%%%%%%%%%%%
\end{appendices}
%%%%%%%%%%%%%%%%%%%%%%%%%%%%%%%%%%%%%%%%%%%%%%%%%%%%

%\clearpage

\bibliographystyle{plain}
\bibliography{sn-bibliography-ADLHC0R}

\end{document}